\documentclass[10pt]{article}
\usepackage{geometry}
\usepackage{comment,hyperref,slashed,tensor}
\usepackage[T1]{fontenc}

\usepackage{ mathrsfs }
\usepackage{amsmath, amsfonts, mathtools, amsthm, MnSymbol}
\usepackage{graphicx, float}
\mathtoolsset{showonlyrefs}

\usepackage{hyperref,  color, xr}

\newtheorem{theorem}{Theorem}[section]
\newtheorem{lemma}[theorem]{Lemma}

\newtheorem{remark}[theorem]{Remark}

\newcommand{\R}{\mathbb{R}}

\newcommand{\N}{\mathbb{N}}

\newcommand{\p}{\partial}

\newcommand{\ringw}{\mathring{w}}
\newcommand{\calH}{\mathcal{H}}
\newcommand{\afm}{\text{afm}}
\newcommand{\bfm}{\text{bfm}}
\newcommand{\aff}{\text{af}}
\newcommand{\bff}{\text{bf}}
\newcommand{\scrl}{\mathscr{L}}
\newcommand{\scrH}{\mathscr{H}}

\title{Invertibility of a linearized Boussinesq flow: a symbolic approach}

\author{Tarek M. Elgindi\thanks{Department of Mathematics, Duke University. E-mail: tarek.elgindi@duke.edu} \space and Federico Pasqualotto\thanks{Department of Mathematics, UC Berkeley. E-mail: fpasqualotto@berkeley.edu}}

\begin{document}

\maketitle

\begin{abstract}
    We develop a computer-assisted \emph{symbolic} method to show that a linearized Boussinesq flow in self-similar coordinates gives rise to an invertible operator.
\end{abstract}

\tableofcontents

\section{Introduction}
A well-established paradigm to construct almost self-similar blow-up solutions to evolution PDEs is to find an approximate simplified model, and then build a solution to the original problem via a \emph{perturbation argument} about the solution of the approximate problem. The properties of the linearization about the approximate solution are thus fundamental for the full construction, as shown in various works (\cite{mrrs, E_Classical} and many other works). The main purpose of the present work is to study a linear problem arising in the context of our analysis of the 2d Boussinesq and 3d Euler equations. In the course of our proof (\cite{EP2023}, Lemma 5.1), we study the linearized operator around an approximate self-similar profile of the Boussinesq equations in self-similar coordinates, and we require that this operator is suitably invertible (up to a one-dimensional kernel). We prove such invertibility statement below.

More precisely, recall that 
\begin{equation}
    \mathfrak{L}f := f + z\p_z f + \scrl f,
\end{equation}
where $\scrl$ was defined in~\cite{EP2023}, Section 4.1. For $k \in \N$, $k \geq 1$, recall the space $\scrH^k$, which was defined in \cite{EP2023}, display~(50). We consider $\mathfrak{L}$ on its natural domain $D(\mathfrak{L})$ whenever the range is $\scrH^k$. With this setup, our goal is to prove the following Theorem.

\begin{theorem}\label{thm:inv} 
Let $k \geq 4$. $ \text{ker}(\mathfrak{L})$ is one-dimensional, and $\text{ker}(\mathfrak{L}) \cap \text{range}(\mathfrak{L}) = \{0\}$.
\end{theorem}

\begin{remark}
    Note that, in~\cite{EP2023}, we gave a softer argument that implies that $\text{ker}(\mathfrak{L})$ is one-dimensional (in the proof of Lemma 5.1). Therefore, the main purpose of the present work is to exclude the presence of a generalized kernel element.
\end{remark}

Upon close inspection of the operator $\mathscr{L}$, the absence of a generalized kernel can be reduced to the absence of kernel of a modified operator $\scrl_0$ acting on 2-vectors of functions of two variables, $(z,\theta)$, which is local in $z$ and nonlocal in $\theta$. This non-local feature constitutes the first main technical obstacle in obtaining the statement of Theorem~\ref{thm:inv}. In addition, the second main technical obstacle lies in the fact that the operator $\mathfrak{L}$ is the linearization around a certain profile, which only possesses fractional regularity at $\theta = \pi/2$ (see Lemma~3.9 in~\cite{EP2023}). Hence, if we wish to express the coefficients of this operator in the form of a global series in the angle $\theta$, the convergence will be very slow (this is, for instance, manifest in Lemma~\ref{lem:spgap} below\footnote{Note that the rate of convergence in this lemma is actually only a lower bound for the actual rate of convergence.}).

There are two main reductions which come into play in our proof. 
\begin{itemize}
    \item First, we notice (in Section~\ref{sec:sufficient}) that a sufficient condition to prove Theorem~\ref{thm:inv} is the positivity (for all times) of a time-dependent quantity which arises from taking the integral average (in the variable $\theta$) of the evolution of a function $F_*(\theta)$ under the semi-group associated to an integro-differential operator \emph{only involving the variable $\theta$}. We call this time-dependent quantity $\Upsilon(t)$.
    
\item Secondly, we show that $\Upsilon(t)$ satisfies a Volterra integral equation \emph{only involving the variable $t$}, whose (time-dependent) coefficients and kernel can be computed to arbitrary precision, provided that we take enough terms in the series expansion of the profile (and of another associated quantity, $w$, appearing in Lemma~\ref{lem:weight}).
\end{itemize}

Note that the above reductions ``transform'' the non-locality in $\theta$ to non-locality in time. We therefore reduce the problem to the analysis of a linear Volterra integral equation of the type:
$$
\Upsilon(t) + \int_0^t K(t-s) \Upsilon(s) ds = g(s),
$$
where $K$ and $g$ are positive functions. Moreover, it can be shown that $K(t)$ and $g(t)$ can be computed using a global series expansion\footnote{After integration in $\theta$, these coefficients then depend only on $t$.} (in $\theta$), however the convergence is polynomial and very slow. As mentioned earlier, this is related (among other things) to the lack of regularity of the profile. 

\subsection{The symbolic approach and comparison with recent literature}

Due to the slow convergence of the expansions in $\theta$, we require a large number of terms in the expansions of $K$ and $g$. This leads us to develop a \emph{symbolic approach} to compute these terms with the aid of a computer. The main advantage of the symbolic method is that we are able to perform all (computer-assisted) calculations on \emph{symbolic fractions}, which we are able to calculate \emph{exactly}. This effectively eliminates the need to keep track of the error bounds via interval arithmetic. It appears that such an approach would be the closest to a non-computer assisted proof. Moreover, in the final step, we require to prove that a certain expression is positive: in view of a number of reductions\footnote{Not limited to the reductions described above in the introduction.}, this is eventually reduced to studying the sign of certain polynomials of high degree with rational (symbolic) coefficients. This can be done exactly with symbolic calculations. 

We remark that what is presented here can certainly be achieved by means of simulating the original angular integro-differential equation, combined with interval arithmetic. We choose to carry out our argument in a way that is closest to an analytical proof\footnote{In addition, we do not exclude that an analytical (``by hand'') proof of the statement in Theorem~\ref{thm:inv} could be obtained implementing a more careful asymptotic analysis.}. Moreover, we highlight that, while the main technical component of our work relies on exact symbolic computations, certain steps within the computer-assisted part effectively implement procedures akin to those found in interval arithmetic or floating-point computation. For instance, the use of the built-in function \texttt{rat} to find simplified rational upper or lower bounds for exact rational numbers, driven by the need for computational efficiency, is analogous to the bounding of ranges in interval arithmetic. This approach allows us to achieve a similar goal of obtaining rigorous bounds on quantities of interest, with a less extensive technical framework than a full implementation of interval arithmetic.

Finally, a comparison is in order with recent advances in the area of computer assisted proofs in fluid dynamics. In the recent remarkable~\cite{ChenHouSmooth1}--\cite{ChenHouSmooth2}, Chen and Hou provide a computer-assisted argument for singularity formation from smooth data for the Boussinesq equation on $\mathbb{R}^2_+$ and the Euler equations on the interior of a cylinder. While the paper includes a significant analytical component, computer assistance is used at almost all levels of the proof in \cite{ChenHouSmooth1}--\cite{ChenHouSmooth2}, from construction of an approximate profile to the proof of coercivity of the associated linearized operator. 

The scope of the computer-assisted analysis in the present paper is significantly narrower and is only used to solve a sub-problem in a situation where \emph{a blow-up scenario has already been identified}. In our case, the computer assistance comes into play when verifying what (in essence) are the hypotheses of a fixed point argument: indeed, we need to verify suitable invertibility of the linearized operator at the endpoint $\alpha = 0$ (compare with the statement of Theorem~\ref{thm:inv}). A purely analytical proof appears to be technically challenging (though we do not rule it out), which leads us to opt for a computer-assisted symbolic method.

Finally, in the direction of computer-assisted proofs in fluid dynamics, let us also mention the recent works \cite{bcg2022, dahne, chen2021}. For a comprehensive discussion of computer-assisted proofs in PDEs, their applications, and interval arithmetic, we refer the interested reader to the survey~\cite{javi_review}. The above comparisons and references are far from complete; see the main paper of this series~\cite{EP2023} for a more detailed introduction and contextualization in the larger topic of singularity formation in incompressible fluids.

We proceed to introduce the main objects studied in this paper, and state the main theorem, from which Theorem~\ref{thm:inv} follows.

\subsection{Setup and main Theorem}

We consider the following integro-differential operator acting on functions $q: [0, \infty)_z \times [0, \pi/2]_\theta$:
$$
\mathscr{L}_0 q := z \partial_{z} q+\frac{1}{1+z} \mathcal{L} q+\frac 12 F_{*} \int_{z}^{\infty} \strokedint \frac{q}{(1+y)^{2}} d y.
$$
Here, whenever $\Gamma = (\Gamma_1, \Gamma_2)$ are functions of $\theta$, we have
$$
\mathcal{L}(\Gamma_1, \Gamma_2) =L_\theta \Gamma - \frac{\Gamma^*}{2} \strokedint \Gamma_1,
$$
and
\begin{equation}
    L_\theta \Gamma =  \Big(2 D_\theta \Gamma_1 + 6 \Gamma_1 -  \Gamma_2, \quad  2 D_\theta \Gamma_2 + 9 \Gamma_2 - 130 \cos^2 \theta \Big(\Gamma_1 - \strokedint \Gamma_1 \Big)  \Big) 
\end{equation}
Moreover $\Gamma^*$ is the profile, which is the unique solution (in the space $\calH^k_\theta$, which was defined in~\cite{EP2023}, Section 3) of the problem $ L_\theta \Gamma^* = 0$ with  $\strokedint \Gamma_1 = 2$,
and $F^*$ is obtained from $\Gamma^*$ as 
\begin{equation*}
F^*_1 = D_\theta \Gamma^*_1 + \frac 5 2 \Gamma^*_1, \qquad \qquad \qquad  F^*_2 = D_\theta \Gamma^*_2 + 4 \Gamma^*_2.
\end{equation*}
Finally, recall that we have, for a function $f:[0, \pi/2] \to \R$ depending only on $\theta$,
$$
D_\theta f = \sin(2\theta)\p_\theta f, \qquad \qquad \qquad  \strokedint f = \frac{2}{\pi}\int_0^{\frac \pi 2} f(\theta) d\theta.
$$
The main goal for the remainder of the paper is to prove the following Theorem.

\begin{theorem}\label{thm:main}
    Let $k \geq 2$, and suppose that $q \in C^k([0, \infty)_z \times [0, \pi/2)_\theta)$, $q$ is bounded on $[0, \infty)_z \times [0, \pi/2]_\theta)$, and $q \in C^{1, \beta} ([0, \infty)_z \times [0, \pi/2]_\theta)$ for some positive $\beta$. Then, $\mathscr{L}_0  q = 0$ implies $q = 0$ identically.
\end{theorem}

First, we deduce Theorem~\ref{thm:inv} from Theorem~\ref{thm:main}.
\begin{proof}[Proof of Theorem~\ref{thm:inv} assuming Theorem~\ref{thm:main}]
    We focus on proving the absence of a non-trival generalized kernel element\footnote{The fact that the kernel is one-dimensional follows from the same reasoning.}. Let $z \p_z f_* \in \text{ker}(\mathfrak{L})$, and suppose that $f^\#$ satisfies 
    $$
    \mathfrak{L}f^{\#} = z \p_z f_*.
    $$
    We define the quantity $q = (q_1, q_2)$ in the following manner. Let $(G^\#_1, P^\#_0, G^\#_2) := \mathcal{B} f^\#$ (where $\mathcal{B}$ was defined in~\cite{EP2023}, Section~2.3), and define\footnote{These quantities arise naturally if one wants to \emph{mod out} the one-dimensional kernel present in $\mathfrak{L}$.}
    \begin{equation}
       Q_1 := G^\#_1 + \frac 1 {B_*} (z\p_z P^\#_0) \Gamma_1^*,\qquad  Q_2:= G^\#_2 - \frac 1 {2B_*^2} z\p_z P^\#_0 \Gamma_2^*, \quad \text{and} \quad  P_0 = P^\#_0.
    \end{equation}
    Here, $B_*^2 = 130$ and $\Gamma_1^*, \Gamma_2^*$ is the angular profile (which satisfies $L_\theta \Gamma^* = 0$). Finally, let
    \begin{equation}
        q_1 : = \frac{(1+z)^2}{z} Q_1 \qquad \qquad q_2 : = \frac{(1+z)^2}{z}Q_2.
    \end{equation}
    Under these conditions, we computed in~\cite{EP2023} (proof of Lemma~5.1) that $q$ satisfies 
    $$
    \scrl_0 q = 0.
    $$ It then follows immediately from the definition of $\scrH^{k}$ and standard Sobolev embeddings that $q$ is in the hypotheses of Theorem~\ref{thm:main} with $k$ replaced by $k-2$. Theorem~\ref{thm:inv} follows.
\end{proof}

\subsection{A sufficient condition for the absence of a generalized kernel element}\label{sec:sufficient}

In this section, we are going to reduce the proof of Theorem~\ref{thm:main} to a statement about the evolution under the semigroup induced by $\mathcal{L}$ of the function $F_*$. More precisely, we have
\begin{lemma}[Reduction to angular averages]\label{lem:first}
    Let $\Upsilon(t) := \strokedint \exp(- t \mathcal{L}) F_* d\theta$. Then Theorem~\ref{thm:main} holds provided that 
    \begin{equation}
         \Upsilon(t) > 0
    \end{equation}
    for all $t \geq 0$.
\end{lemma}

\begin{proof}[Proof of Lemma~\ref{lem:first}]
We consider a function $q$ in the assumption of the Theorem which satisfies
$$
z \partial_{z} q+\frac{1}{1+z} \mathcal{L} q+\frac{F_{*}}{2} \int_{z}^{\infty} \strokedint \frac{q}{(1+y)^{2}} d y=0.
$$
Now define $\eta(z)$ by $z \partial_{z} \eta=\frac{1}{1+z}$, that is $\eta(z)=\log \frac{z}{1+z}.$ It follows that
$$
z \partial_{z} \exp (\eta(z) \mathcal{L}) q+\exp (\eta(z) \mathcal{L}) \frac{F_{*}}{2} \int_{z}^{\infty} \strokedint \frac{q}{(1+y)^{2}} d y=0.
$$
\begin{remark} Note that $\eta(z)<0$ and $\eta(0)=-\infty$ while $\eta(\infty)=0$. Since $\mathcal{L}$ is accretive, this means that $\exp (\eta(z) \mathcal{L})$ vanishes at $z=0$.
\end{remark}
Upon integration,
$$
q+ \int_{0}^{z} w^{-1} \exp ((\eta(w)-\eta(z)) \mathcal{L}) \frac{F_{*}}{2} \int_{w}^{\infty} \strokedint \frac{q}{(1+y)^{2}}=0.
$$
Let us define 
$
\Upsilon(t) := \exp (-t \mathcal{L}) F_{*} d\theta,
$
which yields
$$
Q+\int_{0}^{z} w^{-1} \frac{1}{2} \Upsilon(\eta(z)-\eta(w)) \int_{w}^{\infty} \frac{Q}{(1+y)^{2}} d y d w=0.
$$
where $Q:= \strokedint q d\theta$. We define:
$$
K(w, z)=\int_{0}^{w} \sigma^{-1} \frac{1}{2} \Upsilon(\eta(z)-\eta(\sigma)) d \sigma
$$
Recall that $\eta(z)<0$. We obtain:
$$
Q+\int_{0}^{z} \frac{d}{d w}(K(w, z)) \int_{w}^{\infty} \frac{Q}{(1+y)^{2}} d y d w=0.
$$
Integrating by parts we get:
$$
Q+ K(z, z) \int_{z}^{\infty} \frac{Q}{(1+y)^{2}} d y+ \int_{0}^{z} K(y, z) \frac{Q}{(1+y)^{2}} d y d w=0.
$$
By our assumption on positivity of $\Upsilon$, we have that $K(w,z)$ is increasing in $w$. Let us now assume by contradiction\footnote{Note that our assumptions in particular imply that this integral is finite, and without loss of generality we can assume it is positive.} that $\int_0^\infty \frac{Q}{(1+y)^2} dy < 0$. 
Note that, setting $q = \log\Big(\frac{z}{1+z}\frac{1+\sigma}{\sigma} \Big)$, we have
\begin{equation*}
    K(z,z) = \frac 12 \int_0^\infty \Upsilon(q) \frac{1+z}{(1+z)e^q -z}e^q d q.
\end{equation*}

It thus follows that $Q(0) > 0$, therefore $z_1 := \inf \{z \in \R: Q(z)< 0 \} > 0$. Moreover, $z_1$ is finite. Then, computing everything at $z_1$,
\begin{equation}
    \int_0^{z_1} (K(y,z_1) - K(z_1,z_1))\frac{Q(y)}{(1+y)^2} dy =  -\int_0^\infty \frac{Q}{(1+y)^2} dy >0.
\end{equation}
However, the integral on the LHS is negative, since $K(z_1, z_1) \geq K(w,z_1)$ for $w \leq z_1$, contradiction. This concludes the proof of Lemma~\ref{lem:first}.
\end{proof}

It thus remains to show that $\Upsilon(t) > 0$ for all $t \geq 0$. Recalling the form of $\mathcal{L}$, we have that\footnote{Alternatively, $\Upsilon(t) = \frac 1{13}\Big(\strokedint \Theta_1(t) - \int_0^t e^{-t+s} \strokedint \Theta_1(s) ds \Big)$}
\begin{equation}\label{eq:upsilons}
\Upsilon(t) = \frac{1}{13} \mathscr{F}^{-1} \Big(\frac{\xi}{1+\xi} \strokedint \hat \Theta_1 \Big),
\end{equation}
where $ \mathscr{F}^{-1}$ denotes inverse Laplace transform, and $\hat \Theta_1$ denotes the Laplace transform of $\Theta_1$ in time, where we define
$$
\Theta = \exp(- t L_\theta) F^*.
$$
After setting $\gamma = \tan \theta$, the evolution of $\Theta = (\Theta_1, \Theta_2)$ is determined by the following fundamental linear system\footnote{Note the renormalization by a factor of $13$ in the $\Theta_1$ component.} on $(t, \gamma) \in [0,\infty)_t \times [0, \infty)_\gamma$:
\begin{align}
  &\p_t  \Theta_1 + 4 \gamma \p_\gamma  \Theta_1 + 6  \Theta_1 - 13  \Theta_2 = 0 \label{eq:theta1pre}\\
  &\p_t  \Theta_2 + 4 \gamma \p_\gamma  \Theta_2 + 9 \Theta_2 - \frac{10} {1+\gamma^2}( \Theta_1 - \strokedint  \Theta_1) = 0.\label{eq:theta2pre}
\end{align}
together with the initial conditions\footnote{Due to the renormalization of a factor of $13$ in the $\Theta_1$ component, the profile $(\Gamma_1^*, \Gamma_2^*)$ satisfies $\strokedint \Gamma_1 = 26$, and $\tilde{L}_\theta \Gamma^* =0$, with
\begin{align}
\tilde{L}_\theta \Gamma :=  \Big(4 \gamma \p_\gamma \Gamma_1 + 6 \Gamma_1 -  13 \Gamma_2, \quad  4 \gamma \p_\gamma \Gamma_2 + 9 \Gamma_2 - 10 \cos^2 \theta \Big(\Gamma_1 - \strokedint \Gamma_1 \Big)  \Big).
\end{align}}:
$F_* := (2\gamma\p_\gamma\Gamma^*_{1} + \frac 52 \Gamma^*_{1}, 2\gamma\p_\gamma\Gamma^*_{2} + 4\Gamma^*_{2}).$

The aim of the rest of the paper is to prove the following Theorem.

\begin{theorem} Let $\Theta$ satisfy~\eqref{eq:theta1pre}--\eqref{eq:theta2pre} with initial conditions $\Theta(t=0) = F_*$. Then,
    $\strokedint \mathscr{F}^{-1}(\frac{\xi}{1+\xi} \hat \Theta_1)$
    is strictly positive for all $t > 0$.
\end{theorem}

We briefly outline the proof strategy and the organization of the paper.

\subsection{Strategy of proof and organization of the paper}

The starting point of our analysis is the observation that the evolution of the angular average $\strokedint \Theta_1$ solves a Volterra integral equation, whenever $\Theta$ solves the system~\eqref{eq:theta1pre}--\eqref{eq:theta2pre}. This will be shown in Section~\ref{sec:volterra}. The key computation (which was also central to the analysis in our main paper~\cite{EP2023}), is that, upon defining $$\bar \Theta := \Theta(t,\gamma) - \Theta(t, \gamma =0),$$ the quantity
$$
\mathcal{I}(t) := \int_0^\infty (\bar{\Theta}_1 + {\bar \Theta}_2 )\frac{d\gamma}{\gamma^2}
$$
is invariant under the evolution of the system~\eqref{eq:theta1pre}--\eqref{eq:theta2pre}.

Building upon this observation, we argue by duality as follows. For convenience, let us rewrite the system~\eqref{eq:theta1pre}--\eqref{eq:theta2pre} schematically in the form
\begin{equation}\label{eq:locnloc}
\p_t \bar \Theta + L_{\text{loc}} \bar \Theta = L_{\text{nloc}} \bar \Theta,
\end{equation}
where $L_{\text{nloc}}\bar \Theta := \left(\begin{array}{c}
       0\\
     \frac{10\gamma^2}{1+\gamma^2} \strokedint \bar{\Theta}_1
     \end{array} \right)$, and $L_{\text{loc}} := L_\theta - L_{\text{nloc}}$.
     We test the equation~\eqref{eq:locnloc} in $L^2(\gamma^{-2}d\gamma)$ against a time-dependent weight\footnote{Note that this weight differs slightly from what is called $w$ later in the manuscript.} $w = (w_1, w_2)$, and obtain\footnote{Here, $(\cdot, \cdot)$ denotes the $L^2$ inner product with respect to the measure $\gamma^{-2} d\gamma$.}
     \begin{equation}
         \p_t (w, \bar \Theta)  - (\p_t w, \bar \Theta) + (L_{\text{loc}}^\dagger w, \bar \Theta) = (w, L_{\text{nloc}} \bar \Theta).
     \end{equation}
     We finally impose two conditions on the weight. Let $T$ be a non-negative time, then we impose
     \begin{itemize}
         \item $\p_t w = L^\dagger_{\text{loc}}w$, and
         \item $w(T) = (1,1)$.
     \end{itemize}
     These conditions ensure that the following holds:
     \begin{equation}\label{eq:volterraprelim}
     \int_0^T (w, L_{\text{nloc}}\bar \Theta)(s) ds = (w(T), \bar \Theta(T)) - (w(0), \bar \Theta(0)) = \mathcal{I}(T) - (w(0), \bar \Theta(0)),
     \end{equation}
     Since $\mathcal{I}$ is constant in time, the RHS can be computed in terms of the invariant $\mathcal{I}(t = 0)$, and $w$, which satisfies a certain transport PDE. Hence, since $L_{\text{nloc}}$ only depends on $\strokedint \bar{\Theta}_1$, display~\eqref{eq:volterraprelim} gives the required Volterra integral equation. It is then straightforward to recover $ \Theta(t, \gamma =0)$ from $\strokedint {\bar \Theta}_1$, which yields $\strokedint \Theta_1$. Without going into details, by manipulating~\eqref{eq:volterraprelim} and convolving with a certain kernel, we obtain the following schematic equation for $f(t) := \Upsilon(t)$ (which was defined in \eqref{eq:upsilons}):
     \begin{equation}\label{eq:volterraprebetter}
         f(t) + \int_0^t \kappa(t-s) f(s) ds= g(t),
     \end{equation}
     where $\kappa(t)$ is a smooth, positive kernel and $g(t)$ is a positive function.

     In Section~\ref{sec:profile}, we provide a series expansion of the profile $\Gamma^*$, which is useful in our subsequent calculations.

     In Section~\ref{sec:largetime}, we show that $\Upsilon(t)$  is positive for large enough times. First, it turns out that there is a relatively straightforward $L^\infty$ based argument which implies that, if $\mathcal{I} =0$, the evolution of $\Theta$ tends to $0$ as $t \to \infty$. To exploit this fact, we split $F^*$ (the initial data for $\Theta$) into a multiple of the profile $c \Gamma^*$ plus a part $F^{*, \text{mod}}$ which satisfies $\mathcal{I}[F^{*, \text{mod}}] = 0.$ We call the resulting evolution (arising from $F^{*, \text{mod}}$) $\Theta^{\text{mod}}$. We know from the $L^\infty$ argument that $\strokedint \Theta^{\text{mod}}\to 0$ as $t \to \infty$.
     
     At this point, we recall that
     $$
     \Upsilon(t) = \mathscr{F}^{-1}\Big(\frac{\xi}{1+\xi} \strokedint \Theta_1 \Big) = \strokedint \Theta_1(t) - \int_0^t e^{-t+s}\strokedint \Theta_1(s) ds = c e^{-t} \strokedint \Gamma_1^* - e^{-t} \int_0^\infty e^{s} \strokedint \Theta_1^{\text{mod}}(s) + (\text{error}).
     $$
     Here, the (error) term tends to $0$ faster than $e^{-t}$, and we are able to compute explicit bounds on the rate and the constant. Therefore, the only remaining object to compute is the second term in the last display above, i.e. the Laplace transform at $-1$ of $\strokedint \Theta^{\text{mod}}$, which we compute by using the Volterra equation~\eqref{eq:volterraprebetter}. By careful estimates on the $L^\infty$ norm of $\Theta^{\text{mod}}$, and computing a lower bound for the Laplace transform of $\strokedint \Theta_1$ with good precision using symbolic calculations, we obtain that $\Upsilon(t) > 0$ for all $t \geq \log(4)$ (see Figure~\ref{fig:largetime} for a plot of the lower bound for $\Upsilon(t)$). Note that a similar, purely analytical, argument gives $\Upsilon(t) > 0$ for $t \geq 5$. Hence computer assistance (which enables us to obtain very precise bounds) is used only to reach $\log(4)$.

\begin{figure}[H]
    \centering
    \includegraphics[width=0.8\linewidth]{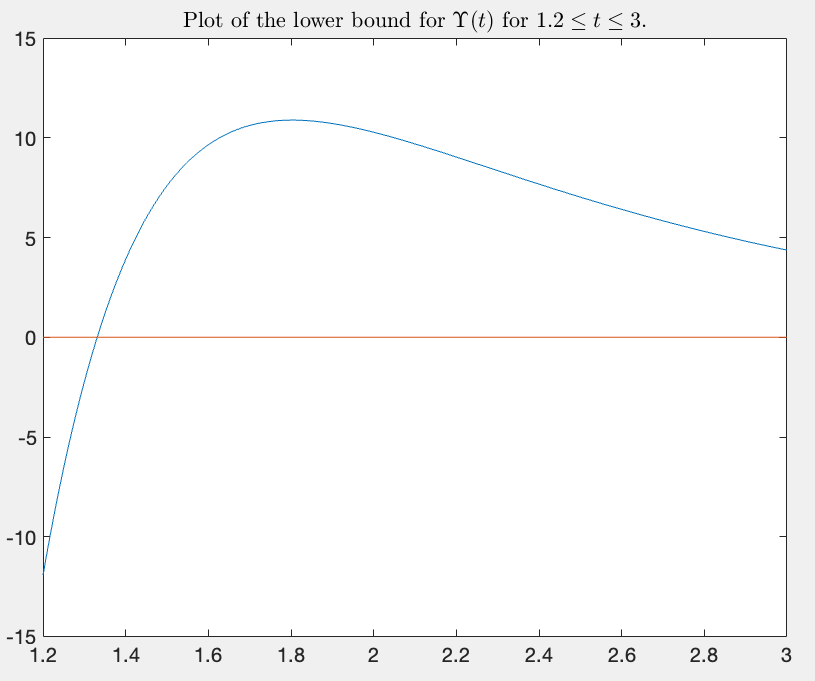}
    \caption{Evolution of the lower bound for $\Upsilon(t)$ for $1.2 \leq t \leq 3$. Note that the lower bound (plotted in blue) becomes positive around time $t = 1.35$, and it stays positive afterwards, also when $t \geq 3$.}
    \label{fig:largetime}
\end{figure}

     It remains to show that $\Upsilon(t) > 0$ for all $t \in [0, \log(4)]$, which we carry out in Section~\ref{sec:shorttime}. To do so, we recall again the Volterra integral equation~\eqref{eq:volterraprebetter}. In this case, we notice that positivity of $g(t)$ and $\kappa(t)$ implies that, as long as $f(t)$ stays positive, the even Picard iterates $P_{2n}(t)$ are upper bounds for $f(t)$, and the odd Picard iterates $P_{2n+1}(t)$ are lower bounds for $f(t)$. We proceed to show that $P_3(t) > 0$ for $t \in [0, \log(4)]$, which concludes the proof (see Figure~\ref{fig:shorttime} for a plot of the lower bound for $P_3(t)$ when $0 \leq t \leq \log(4) + 0.5$). The technical difficulty here is to obtain precise enough bounds on $\kappa(t)$ and $g(t)$. Indeed, these are obtained by expanding each of the corresponding ($\gamma$-dependent) quantities into a global series in powers of $\gamma^2/(1+\gamma^2)$. These series expansions only have polynomial convergence (which is moreover rather slow). Hence, the relevant bounds (for $\kappa$ and $g$) are computed symbolically using a large ($\sim 25$) number of terms.

     \begin{figure}[H]
         \centering
         \includegraphics[width=0.8\linewidth]{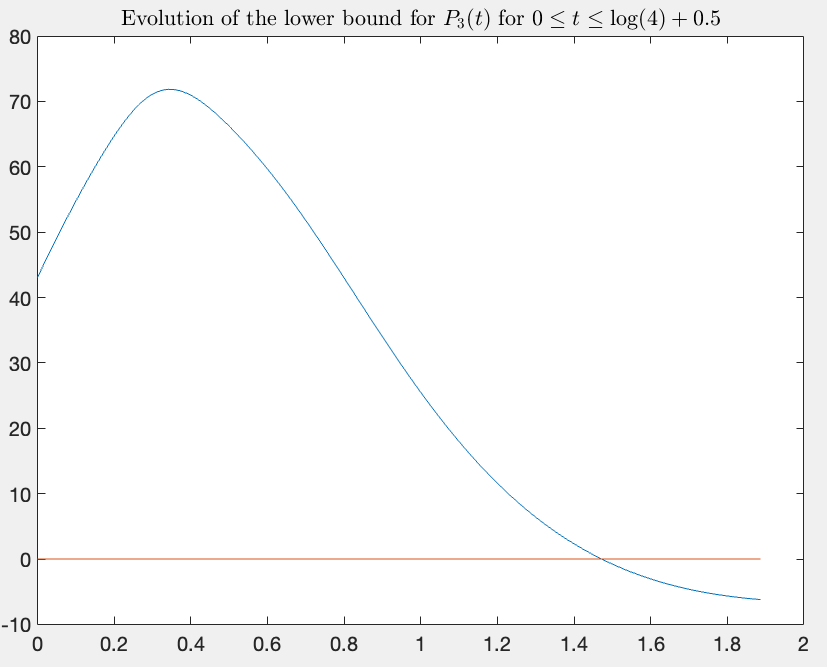}
         \caption{Evolution of the lower bound for $P_3(t)$ for $0 \leq t \leq \log(4) + 0.5$. The lower bound for $P_3(t)$ is in blue, and the orange graph is the constant zero function. Note: the lower bound becomes negative after $t = 1.4$.}
         \label{fig:shorttime}
     \end{figure}

    Finally, the appendix deals with the technical aspects of the proof. In Appendix~\ref{sec:symbolicproof}, we describe in detail the symbolic part of our proof, which is mainly contained in the companion files. In Appendix~\ref{sec:appb}, we prove (again symbolically) the necessary bounds that are then fed into the argument in Sections~\ref{sec:largetime} and Section~\ref{sec:shorttime}. In Appendix~\ref{sec:companion}, 
    we provide a short guide to navigate the companion files. Finally, in Appendix~\ref{sec:appd} we recall some useful facts and lemmas.
    
\subsection{Acknowledgements}
FP thanks Samuel Lanthaler for insightful discussions on the computer assisted part of this work.

\section{Reduction to a Volterra integral equation}\label{sec:volterra}

In this section, we show that the evolution of the angular average $\Upsilon$ is governed by a certain Volterra integral equation. This Volterra equation is instrumental in showing that $\Upsilon(t) > 0$ for time $t \in [\log(4),\infty)$ (since it allows us to compute a certain Laplace transform at $\xi = -1$), and for time $t \in [0,\log(4))$ (since it allows us to compute precisely the evolution up until time $\log(4)$ via Picard iteration). The aforementioned Volterra equation is found through a duality argument, which we describe in the following section.

\subsection{The duality argument}

The aim of this section is to prove a ``duality'' lemma, which will allow us to compute $\Upsilon(t)$ precisely.

First, we introduce a slightly altered version of the fundamental system~\eqref{eq:theta1pre}--\eqref{eq:theta2pre}. We set 
$$\bar \Theta := \Theta - \Theta(\gamma = 0).$$
We have that $\bar \Theta$ satisfies the following fundamental linear system of PDEs:
\begin{align}
  &\p_t \bar \Theta_1 + 4 \gamma \p_\gamma \bar \Theta_1 + 6 \bar \Theta_1 - 13 \bar \Theta_2 = 0 \label{eq:theta1}\\
  &\p_t \bar \Theta_2 + 4 \gamma \p_\gamma \bar \Theta_2 + 9 \bar\Theta_2 - \frac{10} {1+\gamma^2}(\bar \Theta_1 + \gamma^2 \strokedint \bar \Theta_1) = 0.\label{eq:theta2}
\end{align}
together with the initial conditions:
$\bar F_* := (2\gamma\p_\gamma\bar \Gamma_1^* + \frac 52 \bar \Gamma_1^*, 2\gamma\p_\gamma \bar\Gamma_2^* + 4\bar\Gamma_2^*).$ 

From the analysis of the semi-group induced by $L_\theta$ (see~\cite{EP2023}, Section~3), it follows that $\strokedint \Theta_1 (t)$ tends to a finite non-zero limit as $t \to \infty$. It will be convenient, in what follows, to introduce a modified version of $\Theta$ whose evolution under the semigroup induced by $L_\theta$ does tend to $0$ as $t \to \infty$. To that end, we also introduce a modified version of $\Theta$ as follows:
\begin{equation}
    \Theta^{\text{mod}} := \Theta - c \Gamma^*,
\end{equation}
where $c$ is determined such that $ \Theta^{\text{mod}}(t) \to 0$ as $t \to \infty$. This, again according to the results\footnote{Note that the invariant $\mathcal{I}(\Xi)$, due to the renormalization of $\Theta_1$ by $13$, is $\int_0^\infty(\bar{\Xi}_1+\bar{\Xi}_2) \frac{d\gamma}{\gamma^2}$.} in~\cite{EP2023} (Lemma~3.9) is accomplished by choosing $c$ such that the invariant $\mathcal{I}(\Theta^{\text{mod}}(t =0) - c \Gamma^*) = 0$, which yields an exact value of $c = \frac{237}{46}$. We have the following Lemma.

\begin{lemma}\label{lem:weight}
    Let $\Theta = (\Theta_1, \Theta_2)$ solve \eqref{eq:theta1}--\eqref{eq:theta2} with initial conditions $\Theta(t=0) = \Theta^{(0)}$. Then, the following integral equation holds for $\strokedint \bar{\Theta}_1$:
    \begin{equation}\label{eq:masterweight}
    10 \int_0^t \Big(\strokedint \bar \Theta_1(s)d\theta\Big) \kappa(t-s) ds = \frac 2 \pi \mathcal{I}(\Theta^{(0)}) - \frac 2 \pi\int_0^\infty\Big(\bar{\Theta}^{(0)}_1 w_1(t) + \bar{\Theta}^{(0)}_2 w_2(t)  \Big) \frac{d\gamma}{\gamma^2} 
\end{equation}
where 
\begin{align}
    \kappa(r) := \frac 2 \pi \int_0^\infty \Big(\frac{w_2(r)}{1+\gamma^2} \Big) d \gamma, \qquad \qquad \mathcal{I}(\Theta^{(0)}) = \int_0^\infty \Big(\frac{\bar{\Theta}^{(0)}_1}{\gamma^2} + \frac{\bar{\Theta}^{(0)}_2}{\gamma^2} \Big)d\gamma, \label{eq:kappadef}
\end{align}
and $(w_1, w_2)$ satisfy the system (on $(t, \gamma) \in [0,\infty) \times [0, \infty)$)
\begin{align}
    &\p_t w_1 - 4 \gamma \p_\gamma w_1 + 10 w_1 - \frac{10}{1+\gamma^2} w_2 = 0, \label{eq:weight1}\\
    &\p_t w_2 - 4 \gamma \p_\gamma w_2 + 13 w_2 - 13 w_1 = 0. \label{eq:weight2}
\end{align}
with initial conditions $(w_1(t=0), w_2(t=0)) = (1,1)$.
\end{lemma}

\begin{proof}[Proof of Lemma~\ref{lem:weight}]

The proof is achieved by multiplying~\eqref{eq:theta1} by $\gamma^{-2} w_1$, multiplying~\eqref{eq:theta2} by $\gamma^{-2} w_2$, and summing the resulting equations. Upon integration in $\gamma$ and in time $s$ from $0$ to $t$, using the fact that $(w_1, w_2)$ lie in the kernel of the dual of the local terms in $L_\theta$, we conclude.
\end{proof}

\subsection{A modified kernel and its approximation}
A close analysis of the kernel $\kappa$ in Lemma~\ref{lem:weight} reveals that it is positive, however $w_1(\gamma=0)$ and $w_2(\gamma=0)$ are identically $1$, which is not favorable from the point of view of taking Laplace transforms in time (since, later in the argument, we will require calculating the Laplace transform at $\xi = -1$). Hence, we introduce modified weights
%\begin{align}
%    \tilde{w}_1 := \delta_{t=0} + \ringw_1, \label{eq:ringw1def}\\
%    \tilde{w}_2 := \delta_{t=0} + \ringw_2, \label{eq:ringw2def}
%\end{align}
where $\ringw_1$ and $\ringw_2$ satisfy the following:
\begin{align}
    &\p_t \ringw_1 - 4 \gamma\p_\gamma \ringw_1 + 10 \ringw_1 - \frac{10}{1+\gamma^2} \ringw_2 = 130 \exp(-23 t),\label{eq:ringw1}\\
    &\p_t \ringw_2 - 4 \gamma\p_\gamma \ringw_2 + 13 \ringw_2 - 13 \ringw_1 = 130 \exp(-23 t),\label{eq:ringw2}
\end{align}
with initial conditions $\ringw_1(t=0) = - 10 \frac{\gamma^2}{1+\gamma^2}$, $\ringw_2(t=0) = 0$.

Note that the definition of the modified weights $\mathring w_1$ and $\mathring w_2$ is chosen to achieve
\begin{equation}\label{eq:ringwdef}
\mathscr{F}(\mathring{w}_1) = \frac{1}{d_0} \mathscr{F}(w_1) -1, \qquad \qquad \mathscr{F}(\mathring{w}_2) = \frac{1}{d_0} \mathscr{F}(w_2) -1
\end{equation}
Here, $d_0 = \frac{\xi + 23}{(\xi+13)(\xi+10)}$ is the first  coefficient in the series expansion~\eqref{eq:series1} of $\hat w_2$. This allows us to obtain a system for $\mathring{w}_1$ and $\mathring{w}_2$ (equations \eqref{eq:ringw1}--\eqref{eq:ringw2}) with vanishing initial condition at $\gamma = 0$, and decaying forcing on the right hand side; moreover, we prefer to divide by $d_0$ to achieve a simpler expansion for $w_2$, which appears in the kernel $\kappa$ defined above in display \eqref{eq:kappadef}. Finally, recall that, if $g: [0,\infty)$ is a function, we denote $\hat g := \mathscr{F}(g)$ the Laplace transform.

We have the following lemma.

\begin{lemma}\label{lem:w1}
Define a sequence $(c_j(\xi), d_j(\xi))_{j=  0, \ldots}$ of meromorphic functions of $\xi \in \mathbb{C}$ as follows. Let $v_j(\xi) := (c_j(\xi), d_j(\xi))$, and
\begin{equation}
    c_0 = \frac 1{\xi+10}, \qquad d_0  = \frac{\xi+ 23}{(\xi + 13)(\xi+10)}.
\end{equation}
Then, define $v_j$ for $j \in \N$, $j \geq 1$ recursively as
\begin{equation}
    v_j = N_j v_{j-1},
\end{equation}
where $N_j$ is the $2 \times 2$ matrix given by
\begin{equation}
    N_j = \left(
    \begin{array}{cc}
    8j + 10 + \xi     & 0   \\
    -13     & 8j + 13 + \xi
    \end{array}
    \right)^{-1}  \cdot \
    \left(
    \begin{array}{cc}
    8(j-1)     & 10   \\
     0    & 8(j-1)
    \end{array}
    \right).
\end{equation} 
Let $(\mathring{c}_j, \mathring{d}_j)_{k=0, \ldots}$ be defined by $\mathring{c}_0 =\mathscr{F}^{-1}\Big( \frac{c_0}{d_0} -1\Big)$, $\mathring{d}_0 = 0$, and
\begin{equation}
    (\mathring{c}_j, \mathring{d}_j) = \mathscr{F}^{-1}\Big(\frac{1}{d_0} (c_j, d_j)\Big),
\end{equation}
for\footnote{The inverse Laplace transform is taken component-wise here.} $k \geq 1$. Then, each of the functions $\mathring{c}_k, \mathring{d}_k$ decay faster than $\exp(-t)$ as $t \to \infty$, and, defining (whenever $M \in \mathbb{N}$, $N \geq 2$)
    \begin{equation}
        A_M := \sum_{j=1}^M \mathring{c}_j(t)\frac{1}{(1+\gamma^2)^j}, \qquad \qquad B_M := \sum_{j=1}^M \mathring{d}_j(t)\frac{1}{(1+\gamma^2)^j},
    \end{equation}
    the errors $\alpha := \ringw_1 - A_M$, $\beta := \ringw_2 - B_M$ satisfy
\begin{align}
    &\p_t \alpha - 4 \gamma \p_\gamma \alpha + 10 \alpha - \frac{10}{1+\gamma^2} \beta = \frac{1}{(1+\gamma^2)^{M+1}} P_M(t),\\
    &\p_t \beta - 4 \gamma \p_\gamma \beta + 13 \beta - 13 \alpha= \frac{1}{(1+\gamma^2)^{M+1}} Q_M(t),
\end{align}
where $P_M(t) = 8M \mathring{c}_M(t) + 10 \mathring{d}_M(t)$, and $Q_M(t) = 8M \mathring{d}_M(t)$.
\end{lemma}

\begin{proof}[Proof of Lemma~\ref{lem:w1}]
The sequence of Laplace transforms $(c_j, d_j)$ is obtained by expansion of the functions $w_1$ and $w_2$ by matching terms of the same order in inverse powers of $(1+\gamma^2)$ (note that we don't require anything about convergence). We set
\begin{equation}\label{eq:series1}
    \hat w_1 = \sum_{k=0}^M \frac {c_k(\xi)}{(1+\gamma^2)^k}, \quad \hat w_2 = \sum_{k=0}^M \frac {d_k(\xi)}{(1+\gamma^2)^k}.
\end{equation}
Taking (formally) the Laplace transform of system~\eqref{eq:weight1}--\eqref{eq:weight2}
\begin{align}
    &\xi \hat w_1 - 4 \gamma \p_\gamma \hat w_1 + 10 \hat w_1 - \frac{10}{1+\gamma^2} \hat w_2 = 1,\\
    &\xi \hat w_2 - 4 \gamma \p_\gamma \hat w_2 + 13 \hat w_2 - 13 \hat w_1 = 1.
\end{align}
Matching powers of $(1+\gamma^2)^{-j}$ we have, for $j = 0$,
\begin{align}
    &\xi c_0 + 10 c_0  = 1,\label{eq:rec00}\\
    &\xi d_0  + 13 d_0 - 13 c_0 = 1.\label{eq:rec01}
\end{align}
and for $j > 0$,
\begin{align}
    &\xi c_j + 8 j c_j - 8(j-1)c_{j-1} + 10 c_j - 10 d_{j-1} = 0,\label{eq:rec1}\\
    &\xi d_j + 8 j d_j - 8(j-1)d_{j-1} + 13 d_j - 13 c_j = 0.\label{eq:rec2}
\end{align}
This yields the recursion relation for $N_j$. The statements for $(\mathring{c}_j, \mathring{d}_j)$ follow after noticing that
$$
\mathscr{F}(\mathring{w}_1) = \frac{1}{d_0} \mathscr{F}(w_1) -1, \qquad \qquad \mathscr{F}(\mathring{w}_2) = \frac{1}{d_0} \mathscr{F}(w_2) -1
$$
satisfy~\eqref{eq:ringw1}--\eqref{eq:ringw2}, which is a direct computation.
\end{proof} 

We also have the following lemma, which will be instrumental in computing bounds at $\xi = -1$.

\begin{lemma}\label{lem:m1bound}
    When $j \geq 10$, the matrix $N_j(\xi=-1)$ satisfies
    \begin{equation}
        |N_j(\xi=-1)|_\infty \leq \Big(1 - \frac{0.85}{j}\Big).
    \end{equation}
    Here, $|N|_\infty$ denotes the maximum of the row sums of the matrix obtained from $N$ taking absolute values entry-wise.
\end{lemma}

\begin{proof}
    Direct calculation.
\end{proof}

Finally, we have the following lemma, whose proof (using symbolic calculations) is contained in the companion file~\texttt{PQ\_sym.m}, and is desribed in Section~\ref{sec:pqproof} below.
\begin{lemma}\label{lem:PQ}
We have $\beta \geq 0$ for all $t \geq 0$, and in addition
$$
P_M(t) - Q_M(t) \geq 0
$$
for all $t \geq 0$ whenever $M = 25$.
\end{lemma}

\subsection{The main Volterra integral equation}
To conclude this section, we find the main Volterra integral equation satisfied by $\Upsilon(t)$. This is achieved by taking the Laplace transform of equation~\eqref{eq:masterweight}. These considerations will be fundamental in later sections~\ref{sec:largetime} and~\ref{sec:shorttime}.

We let the kernel $$\hat{K}_1 := \frac{(\xi+13)(\xi+10)}{\xi+23}\hat{\kappa}(\xi),$$ where, as usual, $\hat{\kappa}$ denotes the Laplace transform of $\kappa$, and $\kappa$ was defined in display~\eqref{eq:kappadef}. By display~\eqref{eq:ringwdef}, $$K_1 =  \delta_{t=0} + K_{2},$$ where $K_2$ is a positive kernel\footnote{It is not difficult to see inductively that $d_j(t)\geq 0$ for all $t \geq 0$.}. More precisely, $$K_2=\frac{2}\pi \int_0^\infty \ringw_2 \frac{d\gamma}{(1+\gamma^2)},$$ where $\mathring w_1, \mathring w_2$ solve~\eqref{eq:ringw1}--\eqref{eq:ringw2}
with initial conditions $(w_1, w_2) = (-10 \gamma^2/(1+\gamma^2), 0)$.

Upon integration of equations~\eqref{eq:theta1pre}--\eqref{eq:theta2pre} with initial conditions $\Theta(t=0) = F^*$, it is straightforward to see that\footnote{With a slight abuse of notation, we set $\hat \Upsilon :=\frac{\xi}{1+\xi} \strokedint \hat \Theta_1$.}
\begin{align}
    &\hat{\Upsilon}(\xi) =  \frac{\xi}{\xi+1}\Big(\frac{1}{\xi+6} F_{*,1}(0) + \frac {13}{(\xi+6)(\xi+9)} F_{*,2}(0)\Big)+\frac{\xi}{\xi+1}\Big(1-\frac {130}{(\xi+6)(\xi+9)}\Big) \strokedint \mathscr{F}(\bar{\Theta}_1).
\end{align}
Multiplying both sides by $\hat K_1$, and using \eqref{eq:masterweight}, we obtain
\begin{equation}
\begin{aligned}
&\hat K_1 \hat \Upsilon(\xi) = \hat K_1 \frac{\xi}{\xi+1}\Big(\frac{1}{\xi+6} F_{*,1}(0) + \frac {13}{(\xi+6)(\xi+9)} F_{*,2}(0)\Big)\\
&\quad +\frac{\xi}{\xi+1}\Big(1-\frac {130}{(\xi+6)(\xi+9)}\Big) \frac{(\xi+10)(\xi+13)}{\xi+23} \frac 1 {10} \frac 2 \pi \Big(\frac{\mathcal{I}}{\xi} - \int\Big(H_1 \hat w_1 + H_2 \hat w_2  \Big) d\gamma\Big) 
\end{aligned}
\end{equation}
This yields
\begin{equation}\label{eq:masterrhs}
\begin{aligned}
&\hat K_1 \hat \Upsilon(\xi)= \hat K_1 \frac{\xi}{\xi+1}\Big(\frac{1}{\xi+6} F_{*,1}(0) + \frac {13}{(\xi+6)(\xi+9)} F_{*,2}(0)\Big)\\
&\quad +\frac{\xi}{\xi+1}\Big(1-\frac {130}{(\xi+6)(\xi+9)}\Big) \frac 1 {10} \frac{2}\pi \Big( \frac{130}{\xi(\xi+23)} \mathcal{I} - \int\Big(H_1 \hat{\mathring{w}}_1 + H_2 \hat{\mathring{w}}_2  \Big) d\gamma\Big).
\end{aligned}
\end{equation}
Here, $\mathcal{I} = \int_0^\infty (H_1 + H_2) d\gamma$ is the invariant, where we used $H := \frac{\bar{F}^*}{\gamma^2}$.

We collect the above observations in the following lemma, which is obtained by inverse Laplace transform of equation~\eqref{eq:masterrhs}.
\begin{lemma}\label{lem:volterra}
    The function $\Upsilon(t)$ satisfies the Volterra integral equation
    \begin{equation}\label{eq:mainvolte}
        \Upsilon(t) + \int_0^t K_2(t-s) \Upsilon(s) ds = g(t),
    \end{equation}
    where
    \begin{align}
        &K_2(t) = \frac{2}\pi \int_0^\infty \ringw_2 \frac{d\gamma}{(1+\gamma^2)},\\
        &g(t) = \mathscr{F}^{-1}\Big(\hat K_1 \frac{\xi}{\xi+1}\Big(\frac{1}{\xi+6} F_{*,1}(0) + \frac {13}{(\xi+6)(\xi+9)} F_{*,2}(0)\Big)\\
&\qquad \qquad \quad +\frac{\xi}{\xi+1}\Big(1-\frac {130}{(\xi+6)(\xi+9)}\Big) \frac 1 {10} \frac{2}\pi \Big( \frac{130}{\xi(\xi+23)} \mathcal{I} - \int\Big(H_1 \hat{\mathring{w}}_1 + H_2 \hat{\mathring{w}}_2  \Big) d\gamma\Big)\Big).
    \end{align}
\end{lemma}

\begin{remark}\label{rmk:forminusone} By the same reasoning, we also have the following equation for the modified quantities:
\begin{equation}\label{eq:mastermod}
\begin{aligned}
&\hat K_1 \strokedint \hat \Theta^{\text{mod}}_1 = \hat K_1 \Big(\frac{1}{\xi+6} F_{*,1}^{\text{mod}}(0) + \frac {13}{(\xi+6)(\xi+9)} F_{*,2}^{\text{mod}}(0)\Big)\\
&\qquad \qquad -\Big(1-\frac {130}{(\xi+6)(\xi+9)}\Big) \frac 1 {10}\frac 2 \pi \int\Big(H^{\text{mod}}_1 \hat{\mathring{w}}_1 + H^{\text{mod}}_2 \hat{\mathring{w}}_2  \Big) d\gamma.
\end{aligned}
\end{equation}
Here, $\mathcal{I} = \int_0^\infty (H^{\text{mod}}_1 + H^{\text{mod}}_2) d\gamma$ is the invariant, with $H^{\text{mod}}:= \frac{\bar{F}^{*, \text{mod}}}{\gamma^2}$.
\end{remark}

This collects all the relevant information required on the weights $\mathring{w}_1$ and $\mathring{w}_2$ and the Volterra integral equation. We turn to the precise analysis of the profile $\Gamma^*$.

\section{Properties of the profile}\label{sec:profile}

In this section, we derive a series expansion for the profile $\Gamma$ in powers of $\frac{\gamma^2}{1+\gamma^2}$, and we collect some explicit calculations for quantities associated with the profile.

\subsection{Series expansion of the profile}\label{sec:series}
We seek a series expansion of the profile $\Gamma^*$. We recall the profile equation (suppressing the asterisk for ease of notation):
\begin{align}\label{eq:funde1}
    &4 \gamma \p_\gamma \bar \Gamma_1 + 6 \bar \Gamma_1 - 13 \bar \Gamma_2 = 0\\
    &4 \gamma \p_\gamma \bar \Gamma_2 + 9 \bar\Gamma_2 - \frac{10} {1+\gamma^2}(\bar \Gamma_1 + \gamma^2 \strokedint \bar \Gamma_1) = 0.\label{eq:funde2}
\end{align}
We let $A = \strokedint \bar \Gamma_1$.

\begin{remark}
Since $\strokedint \Gamma_1 = 2 \times 13 = 26,$ we have $A = - \frac{351}{19}$.
\end{remark}

We have the following Lemma.

\begin{lemma}\label{lem:profile}
    There exist two sequences of real numbers $a_k$, $b_k$ whose associated series converge absolutely, and such that, defining
    \begin{equation}
        \bar{\Gamma}_1 := \sum_{k =0}^\infty a_k \Big(\frac{\gamma^2}{1+\gamma^2}\Big)^k, \qquad \text{and} \qquad \bar{\Gamma}_2 := \sum_{k =0}^\infty b_k \Big(\frac{\gamma^2}{1+\gamma^2}\Big)^k,
    \end{equation}
    the functions $\bar{\Gamma}_1$ and $\bar{\Gamma}_2$ solve~\eqref{eq:funde1}--\eqref{eq:funde2}, with $\strokedint \bar{\Gamma}_1 = A$.
\end{lemma}

\begin{remark}
    Note that, using this series expansion, we can obtain an alternative proof of the existence of the angular profile $\Gamma_*$, which is the content of Lemma 2.5 in our main paper~\cite{EP2023}.
\end{remark}

\begin{proof}[Proof of Lemma~\ref{lem:profile}]
We find a solution to the above equations as a series in $1/(\eta^2+1)$, where $\eta = \gamma^{-1}$. We start by solving:
\begin{align}
    &-4 \eta \p_\eta \bar \Gamma_1 + 6 \bar \Gamma_1 - 13 \bar \Gamma_2 = 0\label{eq:aa1}\\
    &-4 \eta \p_\eta \bar \Gamma_2 + 9 \bar\Gamma_2 - \frac{10} {1+\eta^2}(\eta^2 \bar \Gamma_1 + A) = 0. \label{eq:aa2}
\end{align}
We then write (formally)
\begin{align}
    \bar \Gamma_1 = \sum_{k \geq 0} \frac{a_k}{(1+\eta^2)^k}, \qquad \qquad 
    \bar \Gamma_2 = \sum_{k \geq 0} \frac{b_k}{(1+\eta^2)^k}.
\end{align}
We first note:
$$
\eta \p_\eta ((1+\eta^2)^{-k}) = - \frac{2k}{(1+\eta^2)^k} + \frac{2k}{(1+\eta^2)^{k+1}}
$$
Matching powers, this implies the following system for $a_k, b_k$. First of all, $a_0 = b_0 =0$. Moreover,
\begin{align}
    &8 a_1 + 6 a_1 - 13b_1 = 0,\\
    &8 b_1 + 9 b_1 - 10 a_1 - 10 A = 0.
\end{align}
That is, $a_1 = \frac{13}{14} b_1$ and $b_1 = \frac{35}{27}A.$ For $k \geq 2$,
\begin{align}
    &8k a_k - 8(k-1) a_{k-1} + 6 a_k - 13 b_k = 0\\
    &8k b_k - 8(k-1) b_{k-1} + 9 b_k - 10 a_k + 10 a_{k-1} = 0.
\end{align}
Writing $v_k := (a_k, b_k)$,
\begin{equation}
    M_k v_k = N_k v_{k-1},
\end{equation}
where
$$
M_k := \left(\begin{array}{cc}
 8k+6    & -13 \\
-10     & 8k+9
\end{array}\right)
\qquad 
\text{and}
\qquad
N_k := \left(\begin{array}{cc}
 8(k-1)    & 0 \\
-10     & 8(k-1)
\end{array}\right)
$$
We have that
\begin{equation}\label{eq:W}
W_k := M_k^{-1}N_k = \left(\begin{array}{cc}
\frac{32 k^{2}+4 k-101}{32 k^{2}+60 k-38} & \frac{26 k-26}{16 k^{2}+30 k-19} 
\\
 -\frac{35}{16 k^{2}+30 k-19} & \frac{16 k^{2}-4 k-12}{16 k^{2}+30 k-19} 
\end{array}\right)
\end{equation}
Moreover, $v_k = (a_k, b_k)$ are obtained from $v_{k-1}$ by $v_k = W_k v_{k-1},$
where $W_k$ is given by formula~\eqref{eq:W}. Lemma~\ref{lem:spgap}, combined with lemma~\ref{lem:prodottino} then gives the claim about uniform convergence. Finally, the fact that $\Gamma$ is a solution to~\eqref{eq:funde1}--\eqref{eq:funde2} follows from the fact that the series converges geometrically at each point $\gamma \in \mathbb{R}$, and the fact that $A = \strokedint \bar \Gamma_1$ follows from summing~\eqref{eq:aa1} to \eqref{eq:aa2}, multiplying the result by $\gamma^{-2}$, and integrating the resulting expression\footnote{This is the calculation which yields the ``invariant''.} from $\gamma = 0$ to $\infty$.
\end{proof}

For a 2D vector, we now denote $|v|_2$ the $\ell^2$ norm. We clearly have that
\begin{equation}
    |v_k|_2 \leq |W_k|_2 |v_{k-1}|_2,
\end{equation}
where $|M|_2$ denotes the spectral norm of $M$. We have the following Lemma.

\begin{lemma}\label{lem:spgap}
For $k \geq 4$, $|W_{k+1}|_2 \leq 1 - \frac{1.1}{k}$. 
\end{lemma}

\begin{proof}
    Direct computation.
\end{proof}

In particular, in view of Lemma~\ref{lem:prodottino}, this implies the following uniform bound. Let $M \geq 5$, then, for $k \geq M$:
\begin{equation}
    \sqrt{|a_k|^2 + |b_k|^2} \leq \Big(\frac{k}{M} \Big)^{-1.1}\sqrt{|a_M|^2 + |b_M|^2}.
\end{equation} 

\subsection{Explicit calculations involving the profile}
Let us collect here some explicit calculation involving the profile $\Gamma$. We have
\begin{equation}
    F^*_1(0) = \frac 52 \Gamma^*_1(0), \qquad F^*_2(0) = 4\Gamma^*_2(0).
\end{equation}
Since $ \strokedint  \Gamma^*_1 = 26$, we have
\begin{equation}
 \Gamma^*_1(0) = \frac{13^2\cdot5}{19}, \qquad \Gamma^*_2(0) = \frac{2 \cdot 3\cdot5 \cdot 13}{19}.
\end{equation}
This implies
\begin{equation}
    F^*_1(0) = \frac{13^2\cdot5^2}{2\cdot19}, \qquad F^*_2(0) = 4\Gamma^*_2(0) = \frac{2^3\cdot3\cdot5 \cdot 13}{19}.
\end{equation}

\section{Analysis of the angular average for large time}\label{sec:largetime}

The goal of this section is to establish the result in Theorem~\ref{thm:main} for time $t \geq \log(4)$. The section is divided into two sub-sections. In section~\ref{sec:linfty}, we will derive a relatively straightforward $L^\infty$ estimate on the evolution of $\Theta^{\text{mod}}_1$ and $\Theta^{\text{mod}}_2$, with initial data $F_*^{\text{mod}}$. In section~\ref{sec:laplacem1}, we will derive a precise estimate for $\strokedint \hat \Theta_1(-1)$. This will finally yield the positivity for time $t \geq \log(4)$ of $\Upsilon(t)$. The reason why these two quantities are important can be readily seen in equation~\eqref{eq:reasonwhy} at the end of this section, where the aforementioned quantities appear naturally.

\subsection{An $L^\infty$ bound for the evolution of the modulated quantities}\label{sec:linfty}
\newcommand{\gbar}{\bar{G}}

We start from the fundamental system~\eqref{eq:theta1pre}--\eqref{eq:theta2pre} and we impose the \emph{modified} initial conditions
\begin{equation}
    \Theta^{\text{mod}}(t=0) = F^{*, \text{mod}} :=  F^* - c \Gamma^*,
\end{equation}
where $c = \frac{237}{46}$. We know from (\cite{EP2023}, Lemma 3.9) that $\Theta^{\text{mod}}$ decays to $0$ exponentially as $t \to \infty$.

We obtain, letting $G_1 := \bar{\Theta}^{\text{mod}}_1/\gamma^2,$ and $G_2 = \bar{\Theta}^{\text{mod}}_2/\gamma^2$,
\begin{align}
&\p_t \gbar_1 + 4 \gamma \p_\gamma \gbar_1 + 14 \gbar_1 = 13 \gbar_2,\\
&\p_t \gbar_2 + 4 \gamma\p_\gamma \gbar_2 + 17 \gbar_2 = \frac{10}{1+\gamma^2} \Big(\gbar_1 + \gamma^2 \strokedint \gbar_1 - \gamma^2 \frac{2}\pi \int G_1 \Big)
\end{align}
This directly yields the following relations for the $L^\infty$ norms:
\begin{align*}
&\p_t |\gbar_1|_\infty + 14 |\gbar_1|_\infty \leq 13 |\gbar_2|_\infty,\\
&\p_t |\gbar_2|_\infty + 17 |\gbar_2|_\infty \leq 10 |\gbar_1|_\infty + 10 \cdot \frac 2 \pi \Big|\int G_1 \Big|.
\end{align*}
Since $\mathcal{I}(F^{*,\text{mod}}) = 0$, we have
\begin{equation}
   \Big| \int G_1 d\gamma \Big|= \exp(-23 t) \Big|\int G_1(t=0) d\gamma\Big|\leq 33 \exp(-23t).
\end{equation}
The above reasoning implies
\begin{align}
    |\gbar_1|_\infty + |\gbar_2|_\infty \leq \exp(-4t) (|\gbar_1(t=0)|_\infty + |\gbar_2(t=0)|_\infty) + \frac{20}{\pi} \int_0^t \exp(-4(t-s)) \Big|  \int G_1 d\gamma(s)\Big| ds,
\end{align}
from which we deduce
\begin{align}
    |\gbar_1|_\infty \leq \exp(-27 t) |\gbar_1(t=0)|_\infty + 13 \int_0^t \exp(-27(t-s))  ( |\gbar_1(s)|_\infty + |\gbar_2(s)|_\infty ) ds
\end{align}
This yields a bound on $|\gbar_1|_{\infty}$. We plug this back into the equation for $\bar{\Theta}^{\text{mod}}_2$, which can be rewritten as
\begin{equation}
      \p_t \bar{\Theta}^{\text{mod}}_2 + 4 \gamma \p_\gamma \bar{\Theta}^{\text{mod}}_2 + 9 \bar{\Theta}^{\text{mod}}_2 = \frac{10 \gamma^2} {1+\gamma^2}(\gbar_1  - \strokedint \gbar_1 + \frac 2 \pi \int G_1),
\end{equation}
so that
\begin{equation}
    \p_t |\bar{\Theta}^{\text{mod}}_2|_\infty + 9 |\bar{\Theta}^{\text{mod}}_2|_{\infty} \leq 20|\gbar_1|_\infty + \frac{20}\pi \Big| \int G_1 d\gamma\Big|,
\end{equation}
which implies a bound on $|\bar{\Theta}^{\text{mod}}_2|_\infty$ upon integration in time.
Moreover,
\begin{equation}
    \p_t |\bar \Theta^{\text{mod}}_1|_\infty + 6 |\bar \Theta^{\text{mod}}_1|_\infty \leq 13 |\bar{\Theta}^{\text{mod}}_2|_\infty.
\end{equation}
Finally,
\begin{equation}
    \p_t |\Theta^{\text{mod}}_2(\gamma=0)| + 9 |\Theta^{\text{mod}}_2(\gamma=0)| \leq 10 |\bar \Theta^{\text{mod}}_1|_\infty,
\end{equation}
and this implies a bound on $|\Theta^{\text{mod}}_2(\gamma=0)|$, which in turn, using the inequality
\begin{equation}
    \p_t |\Theta^{\text{mod}}_1(\gamma=0)| + 6 |\Theta^{\text{mod}}_1(\gamma=0)| \leq 13 |\Theta^{\text{mod}}_2(\gamma=0)|
\end{equation}
implies a bound on $|\Theta^{\text{mod}}_1(\gamma=0)|$.

Finally, we have
\begin{equation}
    \Big| \strokedint \Theta^{\text{mod}}_1\Big| \leq |\Theta^{\text{mod}}_1(\gamma = 0)| + |\bar{\Theta}^{\text{mod}}_1|_\infty,
\end{equation}
from which we obtain a bound on the average of $\Theta^{\text{mod}}_1$. The reasoning which we just concluded allows us to deduce an upper bound of $|\Theta^{\text{mod}}_1|$, and all the relevant calculations are conducted (symbolically) in the file \texttt{profile.m}. We record the result in the following Lemma.

\begin{lemma}
    The evolution of $F^{*, \text{mod}}$ under the semigroup given by $L_\theta$ obeys the $L^\infty$ bound in Section A of the companion file \emph{\texttt{profile.m}}.
\end{lemma}

\begin{remark}
Note in particular that the above reasoning shows that the evolution under $L_\theta$ of the initial conditions $F^{*, \text{mod}}$ decays to $0$ exponentially. This is an easier proof of the same statement contained in our main paper~\cite{EP2023}, however, in that paper we required estimates on an $L^2$ based space, which is more intricate. Note that we could have also adapted the $L^2$ based estimates to the present case, however the fact that $L^\infty$ norms are comparatively smaller gives a (purely) computational advantage in this case.
\end{remark}

\begin{remark}\label{rmk:derivative}
    Note also that the above reasoning gives a bound for the derivative $\Upsilon'(t)$, whose proof is carried out in the companion file \emph{\texttt{derivative\_sym.m}}. We obtain that, for $t \geq 0$,
    \begin{align}
    &|\Upsilon'(t)| \leq 20000 \qquad \text{when} \qquad t \in [0, \log(3)],\\
    &|\Upsilon'(t)| \leq 5000 \qquad  \ \text{when} \qquad t \in [\log(3), \log(4)].
\end{align}

The above bounds are obtained from the same argument as in the beginning of Section~\ref{sec:linfty}, after applying $\partial_t$ to the equations~\eqref{eq:theta1pre}--\eqref{eq:theta2pre}. To that effect, we compute the initial conditions $\partial_t \Theta (t = 0) = - L_\theta F_*$ 
(note that these initial conditions are obtained by formally plugging in $F_*$ in the RHS of equations~\eqref{eq:theta1pre}--\eqref{eq:theta2pre}). Importantly, the argument requires $L^\infty$ bounds for $L_\theta F_*$, which are obtained in Section~\ref{sec:linftyapp}. The relevant calculations needed to conclude the full proof are carried out in the companion file \emph{\texttt{derivative\_sym.m}}.

\end{remark}

\subsection{Computing the Laplace transform at $-1$ and its error bounds}\label{sec:laplacem1}

This section concerns the calculation of $\hat \Upsilon(-1)$. We have, from equation~\eqref{eq:mastermod} in Remark~\ref{rmk:forminusone}, that
\begin{equation}\label{eq:mqm}
\begin{aligned}
&\hat K_1(-1) \strokedint \hat \Theta^{\text{mod}}_1(-1)\\
& \quad = \hat K_1(-1) \Big(\frac{1}{5} F_{*,1}^{\text{mod}}(0) + \frac {13}{40} F_{*,2}^{\text{mod}}(0)\Big)+ \frac 9 {40}\frac 2 \pi \int\Big(H^{\text{mod}}_1 \hat{\mathring{w}}_1(-1) + H^{\text{mod}}_2 \hat{\mathring{w}}_2(-1) \Big) d\gamma.
\end{aligned}
\end{equation}It thus behooves us to find bounds for the expression appearing in~\eqref{eq:mqm}. We have the following Lemma.

\begin{lemma}\label{lem:bdmqm} The expressions
\begin{equation}
    \frac 2 \pi \int_0^\infty \Big(H_1 \hat{\mathring{w}}_1(-1) + H_2 \hat{\mathring{w}}_2(-1) \Big)d\gamma,\qquad \text{ and } \qquad \frac2 \pi \int_0^\infty \frac{1}{1+\gamma^2} \hat{\mathring{w}}_2(-1) d\gamma
\end{equation}
satisfy the bounds in the companion file \emph{\texttt{profile.m}}, section D.
    \end{lemma}

\begin{proof}[Proof of Lemma~\ref{lem:bdmqm}]
We define, whenever $k \geq 1$ and $j \geq 0$ are integers,
\begin{equation}
    c_k^{(j)} := \frac{2} \pi \int_0^\infty \Big(\frac{\gamma^2}{1+\gamma^2} \Big)^k \frac{1}{(1+\gamma^2)^j} \frac{d\gamma}{\gamma^2} =  \frac{2} \pi \frac{\Gamma(1/2 + j) \Gamma( k- 1/2)}{2 \Gamma(j + k)},
\end{equation}
where, here, $\Gamma$ denotes the standard Gamma function. The series expansion for $\Gamma^*$ in Lemma~\ref{lem:profile} induces a series expansion for $F^{*, \text{mod}}$ with coefficients $(\afm_k, \bfm_k)$:
\begin{equation}
    \bar{F}^{*, \text{mod}}_1 := \sum_{k =1}^\infty \text{afm}_k \Big(\frac{\gamma^2}{1+\gamma^2}\Big)^k, \qquad \text{and} \qquad \bar{F}^{*, \text{mod}}_2 := \sum_{k =1}^\infty \bfm_k \Big(\frac{\gamma^2}{1+\gamma^2}\Big)^k,
\end{equation}

Indeed, using the definition of $F^{*, \text{mod}}$ and the equations for $\Gamma^*$, we can compute the explicit form of the coefficients $\text{afm}_k$ and $\text{bfm}_k$ as follows:
\begin{align*}
&\text{afm}_k = - \frac 12 a_k + \frac{13}{2} b_k - c a_k \qquad (k \geq 1),  \\
&\text{bfm}_1 = 5 a_1 + 5 A - \frac 12 b_1 - c b_1,\\
&\text{bfm}_k =5 (a_k - a_{k-1}) - \frac 12 b_k - c b_k,\qquad (k \geq 2).
\end{align*}
(here, recall that $c = \frac{237}{46}$ and $A  = - \frac{351}{19} $).

This, on the other hand, implies
\begin{equation}
    \frac 2\pi \int_0^\infty \Big(H_1 \hat{\mathring{w}}_1(-1) + H_2 \hat{\mathring{w}}_2(-1) \Big) = \sum_{k \geq 1} \sum_{j \geq 0} (\text{afm}_k c_j(-1) c_k^{(j)}+ \text{bfm}_k  d_j(-1)c_k^{(j)}).
\end{equation}
We let $r_{kj}:= (\text{afm}_k c_j(-1) c_k^{(j)}+ \text{bfm}_k  d_j(-1)c_k^{(j)})$.

In order to compute the error terms in this expression, note that the following bounds follow from Lemma~\ref{lem:m1bound}, the form of $c_k^{(j)}$, and Lemma~\ref{lem:prodottino}. We fix $N \geq 10$, and we let $m_N :=  \max\{|c_N(-1)|, |d_N(-1)|\}$. 

First, we have
\begin{equation}\label{eq:m1one}
    \sum_{j \geq N}| c_k^{(j)} d_j(-1)| \leq \sum_{j \geq N}c_k^{(j)}\Big(\frac{j}{N} \Big)^{-0.85} m_N,
\end{equation}
We use
\begin{equation}
    c_k^{(j)} \leq \Big(\frac{j+k}{N+k-1} \Big)^{-k + \frac 12} c_k^{N}
\end{equation}
in order to obtain
\begin{equation}\label{eq:horizontal1}
    \sum_{j \geq N}| c_k^{(j)} d_j(-1)| \leq 1.5 \frac{N+k-1}{k + \frac 32} c_k^{(N)} m_N.
\end{equation}
Moreover, we also use
\begin{equation}
c^{(j)}_k \leq \Big(\frac{k+j-1}{N+j-1} \Big)^{- j - \frac 12} c_N^{(j)}.
\end{equation}
This yields
\begin{align}
    &\sum_{k \geq N} \sum_{j \geq N}| a_k c_k^{(j)} d_j(-1)| \leq \sum_{k \geq N} \Big(1.5 \frac{N+k-1}{k + \frac 32} c_k^{(N)} m_N\Big) \\
    &\leq 1.5 m_N c_N^{(N)} \sum_{k \geq N} \Big(\frac{k+N-1}{2N-1} \Big)^{- N - \frac 12}  \frac{N+k-1}{k + \frac 32} \label{eq:restofmatrix}\\
    &\leq 7.5 m_N c_N^{(N)}.
\end{align}
Similar bounds hold for the analogous expression with $c_j(-1)$.

With these bounds in mind, we are ready to estimate the expression, letting $r_{kj}$ be the expression inside the summation,
\begin{equation}\label{eq:111}\tag{$\star$}
\begin{aligned}
    &\frac{2}\pi\int\Big(H_1 \hat{\mathring{w}}_1(-1) + H_2 \hat{\mathring{w}}_2(-1) \Big) = \sum_{k \geq 1} \sum_{j \geq 0} (\text{afm}_k c_j(-1) c_k^{(j)}+ \text{bfm}_k  d_j(-1)c_k^{(j)})\\
    &=\underbrace{\sum_{k \geq 1} \sum_{0 \leq j \leq 35} r_{kj}}_{(1)} + \underbrace{\sum_{35 \leq j \leq N_1} r_{1j}}_{(2)} + \underbrace{\sum_{j > N_1} r_{1j}}_{(3)} + \underbrace{\sum_{2 \leq k \leq 35} \sum_{35 \leq j \leq N_2} r_{kj}}_{(4)}+ \underbrace{\sum_{2 \leq k \leq 35} \sum_{j > N_2} r_{kj}}_{(5)} + \underbrace{\sum_{\substack{k > 35\\j>35}} r_{kj}}_{(6)}.
\end{aligned}
\end{equation}

 We set $N_1 = 5000$, $N_2 = 500$ in \eqref{eq:111}. Now, (1) is computed directly from our knowledge of $I^{F_*, \text{mod}}_j, J^{F_*, \text{mod}}_j $, for which we have rigorous bounds from Section~\ref{sec:IJ} in the Appendix (see Lemma~\ref{lem:IJmod}). (2) and (4) are computed symbolically. The remaining error terms are bounded using the inequalities~\eqref{eq:m1one}, \eqref{eq:horizontal1}, \eqref{eq:restofmatrix}. A similar reasoning applies to $ \frac2 \pi \int_0^\infty \frac{1}{1+\gamma^2} \hat{\mathring{w}}_2(-1) d\gamma$. 
\end{proof}

We proceed to compute a lower bound for $\strokedint \hat{\Theta}^{\text{mod}}_1(-1)$. We have: 
\begin{align}\label{eq:reasonwhy}
    &\strokedint \hat{\Theta}^{\text{mod}}_1(-1)  = \frac{1}{5} F^{*, \text{mod}}_1(0) + \frac {13}{40} F^{*, \text{mod}}_2(0) + (\hat{K}_1(-1))^{-1}  \frac 9{40}  \frac 2 \pi  \int\big( H^{\text{mod}}_1 \hat{\mathring{w}}_1(-1) + H^{\text{mod}}_2 \hat{\mathring{w}}_2(-1)\big)\\
    &\qquad \geq \frac{1}{5} F^{*, \text{mod}}_1(0) + \frac {13}{40} F^{*, \text{mod}}_2(0)+ \frac{1}{(1+\hat{K_2}(-1))_\ell} \frac 9{40}  \frac 2 \pi \int\big( H^{\text{mod}}_1 \hat{\mathring{w}}_1(-1) + H^{\text{mod}}_2 \hat{\mathring{w}}_2(-1)\big)_\ell.
\end{align}
Here, the subscript $\ell$ denotes a lower bound (note, importantly, that the expression in the last parenthesis is negative).

A symbolic calculation finally yields the lower bound, which we record in the following Lemma.
\begin{lemma}\label{lem:laplacemenouno} We have
\begin{equation}
    \strokedint \hat{\Theta}^{\text{mod}}_1(-1) \geq -45.6.
\end{equation}
\end{lemma}
The above calculations allow us to conclude the analysis for large times, which we carry out in the next section.

\subsection{Concluding the analysis for large time}

Finally, we rewrite the evolution under the full operator as
\begin{align}
    &\Upsilon(t) = \mathscr{F}^{-1}\Big( \frac{\xi}{1+\xi} \strokedint \Theta_1\Big) = (26c  + \strokedint \hat{\Theta}^{\text{mod}}_1(-1))e^{-t} + \strokedint \Theta^{\text{mod}}_1 (t) - \int_t^\infty e^{-t+s}\strokedint \Theta^{\text{mod}}_1 (s) ds\\
    & \qquad \geq (26 c  +\Big( \strokedint \hat{\Theta}_1(-1)\Big)_\ell)e^{-t} - \Big|\strokedint \Theta^{\text{mod}}_1 (t) \Big| - \int_t^\infty e^{-t+s}\Big|\strokedint \Theta^{\text{mod}}_1 (s)\Big| ds\label{eq:finalupsilon}
\end{align}
and the last two terms are estimated with the bounds arising from Section~\ref{sec:linfty}. We obtain the following lemma (whose proof is contained in the companion file \texttt{longtime\_sym.m}).

\begin{lemma} \label{lem:longtime} $\Upsilon(t) > 0$ for all $t \geq \log(4)$.
\end{lemma}

\begin{proof}[Proof of Lemma~\ref{lem:longtime}]
    The proof of the lemma is carries out in the companion file \texttt{longtime\_sym.m}, and it is described in section~\ref{sec:prooflongtime} in the Appendix.
\end{proof}
\section{Analysis of the angular average for short time}\label{sec:shorttime}
Aim of this section is to prove the following lemma:

\begin{lemma}\label{lem:shorttime}
$\Upsilon(t) > 0$ for all $t \in [0,\log(4)]$.
\end{lemma}
The proof of the Lemma will reduce to showing that a certain Picard iterate of equation~\eqref{eq:mainvolte} is positive. We will check this positivity symbolically in the companion files.

We are ready to state and prove the main lemma of this section. The rest of the section is then going to be devoted to showing (via symbolic calculations) the lemma.
\begin{lemma}\label{lem:picard}
    Consider the Picard iterates
    \begin{equation}
        P_n(t) := \mathscr{F}^{-1}\Big(\Big( \sum_{j = 0}^n (-1)^j (\hat{K}_2)^j\Big) \hat K_1 \hat{\Upsilon} \Big).
    \end{equation}
    Then, $P_0(0) > 0$ and $P_3(t) > 0$ for all $t \in [0, \log(4)]$.
\end{lemma}

\begin{remark}
    Note that the above Picard iterates satisfy the recursion relation:
    \begin{equation*}
        P_0(t) = g(t), \qquad P_{k+1}(t) = g(t) - \int_0^t K_2(t-s) P_{k}(s).
    \end{equation*}
\end{remark}

Assuming the lemma, we are now going to show that the full evolution of the angular average is positive for times $t \in [0, \log(4)]$.

\begin{proof}[Proof of Lemma~\ref{lem:shorttime}]
We recall the Volterra integral equation from Lemma~\ref{lem:volterra}:
\begin{equation}\label{eq:volterrasimple}
    f(t) + \int_{0}^t K_2(t-s) f(s) ds = g(t),
\end{equation}
where $f(t) = \mathscr{F}^{-1}\Big(\frac{\xi}{1+\xi} \strokedint \hat{\Theta}_1  \Big)$, and $g(t)$ is the inverse Laplace transform of the RHS in~\eqref{eq:masterrhs}. We also know that $K_2(t) \geq 0$ for all $t \geq 0$\footnote{This follows from the fact that $K_2$ is obtained by summing positive terms. These terms are seen to be positive since they are iterated convolutions of positive functions).}.

Since $f(0) = g(0) > 0$, and let us assume by contradiction that $f(t) \leq 0$ somewhere in $[0, \log(4)]$. We let $T^*:= \inf\{t \in [0, \log 4]: f(T^*) \leq 0\}$, and note that $T^* > 0$. Then, on $[0, T^*)$, $f(t) > 0$. Inductively, this implies that 
$$
f(t) \leq P_{2k}(t), \qquad f(t) \geq P_{2k+1}(t)
$$
for all $k \in \N$, and all $t \in [0, T^*)$. 

Indeed, for the base case, since $P_0(t) = g(t)$, and $K_2$, $f$ are positive in the interval considered, $f(t) = g(t) - \int_0^t K_2(t-s) f(s) ds  \geq g(t)$. Assuming then that $f(t) \leq P_{2k}(t)$ for all $t \in [0, \log(4)]$, we have
\begin{equation*}
    f(t) = g(t) - \int_{0}^t K_2(t-s) f(s) ds \geq g(t) - \int_{0}^t K_2(t-s) P_{2k}(s) ds = P_{2k+1}(t)
\end{equation*}
in the interval $t \in [0, \log(4)]$. The step to go from $2k+1$ to $2k +2$ is analogous.

By Lemma~\ref{lem:picard}, we then have $f(T^*)>0$, however $f(T^*) = 0$ by definition of $T^*$, contradiction.
\end{proof}

We are ready for the proof of Lemma~\ref{lem:picard}. We divide the proof in the next few subsections. The main idea is to derive precise enough bounds for $\ringw_1$ and $\ringw_2$, which then imply precise bounds on $K_2$ and $g(t)$ in equation~\eqref{eq:volterrasimple}, from which we can then compute the Picard iterates. Section~\ref{sec:etas} will be devoted to finding such bounds. 

\subsection{Precise bounds for $\mathring{w}_1$ and $\mathring{w}_2$}\label{sec:etas}
In this section, we are going to use precise upper and lower bounds on $\mathring w_1$ and $\mathring w_2$ which can be computed symbolically. Recall the definitions of $A_M, B_M, \alpha, \beta$ from Lemma~\ref{lem:w1} and the lemma~\ref{lem:PQ} on the positivity of $P_M - Q_M$.

We are going to estimate $\alpha$ and $\beta$. We have, after going to Lagrangian coordinates $\gamma_0$ such that $\gamma = \exp(-4t) \gamma_0$, using the Lemma,
\begin{equation}
    (\alpha - \beta)' + 23 (\alpha - \beta)  \leq \frac{1}{1 + e^{-8t} \gamma^2_0} (P_M(t) - Q_M(t)),
\end{equation}
hence\footnote{Mote that, since $M \geq 2$, the initial conditions for $\alpha$ and $\beta$ are trivial.}
\begin{equation}
    \alpha(t) - \beta(t) \leq \underbrace{\int_0^t \exp(-23(t-s)) (P_M-Q_M)(s) \frac{1}{(1+e^{8(t-s)}\gamma^2)^{M+1}}ds}_{:= B_1(t,\gamma)}.
\end{equation}
This implies
\begin{equation}
\p_t \beta \leq 13 B_1(t, e^{-4t}\gamma_0) + \frac{Q_M(t)}{(1+e^{-8t} \gamma_0^2)^{M+1}}    
\end{equation}
from which we have
\begin{align}
    \beta(t) \leq \underbrace{13 \int_0^t \int_0^w \exp(-23(w-s)) (P_M- Q_M)(s) \frac{1}{(1+e^{8(t-s)}\gamma^2)^{M+1}} ds dw + \int_0^t \frac{Q_M(w)}{(1+ e^{8(t-w)}\gamma^2)^{M+1}}dw}_{:= B_2(t, \gamma)}.
\end{align}
which also implies
\begin{equation}
    \alpha(t) \leq B_1(t, \gamma) + B_2(t, \gamma).
\end{equation}
The lower bounds are obtained by a similar reasoning. Collecting the above observations, we obtain the following lemma.

\begin{lemma}\label{lem:alphabeta}
The functions $\alpha$, $\beta$ satisfy the following bounds:
\begin{align} 
    &\int_0^t \exp(-10(t-s)) \frac{P_M(s)}{(1+e^{8(t-s)}\gamma^2)^{M+1}}ds \leq \alpha(t) \leq B_1(t, \gamma) + B_2(t, \gamma),\\
    &\int_0^t \exp(-13(t-s)) \frac{Q_M(s)}{(1+e^{8(t-s)}\gamma^2)^{M+1}} ds\leq \beta(t) \leq  B_2(t, \gamma).
\end{align}
\end{lemma}

\begin{remark}\label{rmk:alphabetalu}
    It is convenient to denote the lower and upper bounds arising in the previous lemma as follows:
    \begin{align}
    &\alpha_\ell \leq \alpha(t) \leq \alpha_u,\\
    &\beta_\ell\leq \beta(t) \leq  \beta_u.
    \end{align}
\end{remark}

Since all relevant error bounds in the Volterra equation~\eqref{eq:mainvolte} are computed in terms of the bounds multiplied by powers of $\frac{\gamma^2}{1+\gamma^2}$, it behooves us to find bounds for the following expressions, when $r \geq 0$, and $M, N \in \N$,
\begin{equation}
    \eta(r,M,N) := \int_0^\infty \frac{1}{(1+e^{8r}\gamma^2)^{M}} \Big(\frac{\gamma^2}{1+\gamma^2}\Big)^N \frac 1 {\gamma^2} d\gamma.
\end{equation}
We collect the bounds in the following Lemma
\begin{lemma}\label{lem:etas}
    We have the following bounds for the expression $\eta$, whenever $M \geq 1$, $N \geq 0$ and $r \geq 0$:
    \begin{align}
         &\frac{\pi} 2 \exp(-4 r) c^{(M)}_1\leq \eta(r,M,1)\leq \frac \pi 2 \exp(-4 r) c^{(M-1)}_1,
    \end{align}
    if $M > N$,
    \begin{align}
          \frac{1}{2^M 2^N}\exp(-8(M+1)r)\frac{1}{2M+1} \leq \eta(r,M,N)\leq  \exp(-8r(N-1)-4r)\frac{\pi}{2} c^{(M-N)}_N + \exp(-8Mr)\frac{1}{2M-1},
    \end{align}
    if $M \leq N$, 
        \begin{align}
        \frac{1}{2^M 2^N}\exp(-8(M+1)r)\frac{1}{2M+1} \leq \eta(r,M,N)\leq  \exp(-8r(M-1)-4r)\frac{\pi}{2} c^{(0)}_M + \exp(-8Mr)\frac{1}{2M-1}.
    \end{align}
\end{lemma}

\begin{remark}
    It will be convenient to adopt the following notation for the upper bounds of the expressions $\eta$ appearing in the lemma above:
    \begin{equation}
        \eta_\ell(r,M,N) \leq \eta(r,M,N) \leq \eta_u(r,M,N).
    \end{equation}
\end{remark}

\begin{proof}
    The proof of the first inequality is straightforward after noting that, since $r \geq 0$,
    \begin{equation}
        \frac{1}{(1+e^{8r}\gamma^2)^{M+1}}\leq \frac{1}{1+\gamma^2} \frac{1}{(1+e^{8r}\gamma^2)^M} \leq \frac{1}{(1+e^{8r}\gamma^2)^M}.
    \end{equation}
    Concerning the second inequality (upper bound), we have
    \begin{equation}
    \begin{aligned}
    & \eta(r,M,N) = \int_0^{\infty} \frac{1}{\left(1+e^{8 r} \gamma^2\right)^{M}}\left(\frac{\gamma^2}{1+\gamma^2}\right)^N \frac{1}{\gamma^2}d\gamma \\
    = & \int_0^1 \frac{1}{\left(1+e^{8 r}\right)^{M}}\left(\frac{\gamma^2}{1+\gamma^2}\right)^N \frac{1}{\gamma^2}+\int_1^{\infty} \frac{1}{\left(1+e^{8 r} \gamma^2\right)^{M}} \left(\frac{\gamma^2}{1+\gamma^2}\right)^N \frac{1}{\gamma^2} d \gamma \\
    \leq & \int_0^1 \frac{\gamma^{2 N-2}}{\left(1+e^{8 r} \gamma^2\right)^{M}} d \gamma+\int_1^{\infty} \frac{1}{\left(1+e^{8 r} \gamma^2\right)^{M}} \frac{1}{1+\gamma^2} d \gamma \\
    \leq & e^{-4 r} e^{-8 r(N-1)} \int_0^{\infty} \frac{\gamma^{2 N-2}}{\left(1+\gamma^2\right)^{M}} d \gamma+e^{-4 r} \int_1^{\infty} \frac{1}{\left(1+\gamma^2\right)^{M}} d \gamma \\
    = & e^{-4 r} e^{-8 r(N-1)}  \frac\pi 2 c^{(M-N)}_N+e^{-4 r}\frac{1}{2 M-1}.
    \end{aligned}
    \end{equation}
    Concerning the lower bound, we have
    \begin{align}
    & \int_0^{\infty} \frac{1}{\left(1+e^{8 r} \gamma^2\right)^M}\left(\frac{\gamma^2}{1+\gamma^2}\right)^N \frac{1}{\gamma^2} d \gamma 
    \geq \int_1^{\infty} \frac{1}{\left(1+e^{8 r} \gamma^2\right)^M} \frac{1}{1+\gamma^2} 2^{-N+1} d \gamma \\
    \geq & \int_1^{\infty} \frac{1}{\left(1+e^{8 r} \gamma^2\right)^{M+1}} 2^{-N+1} d \gamma 
    = \int_{e^{4 r}}^{\infty} \frac{1}{\left(1+\gamma^2\right)^{M+1}} d \gamma 2^{-N+1} e^{-4 r} \\
    \geq & \int_{e^{4 r}}^{\infty} \gamma^{-2(M+1)} d \gamma 2^{-(N+M)} e^{-4 r} = \frac{1}{2^M 2^N}\exp(-8(M+1)r)\frac{1}{2M+1}.
    \end{align}
    The case $M \leq N$ is analogous.
\end{proof}
Before we proceed to estimate the terms in the Volterra integral equation~\eqref{eq:volterrasimple}, we are going to provide precise estimates useful for the terms involving $H_1$ and $H_2$.

\subsection{Bounds for the terms involving $H_1$ and $H_2$}
The series expansion for $\Gamma^*_1$ and $\Gamma^*_2$ induces an analogous series expansion for $H_1$ and $H_2$ with coefficients $\aff_k$, $\bff_k$ as follows:
\begin{equation}\label{eq:afbf}
    H_1 = \gamma^{-2} \sum_{k \geq 1} \aff_k \Big( \frac{\gamma^2}{1+\gamma^2}\Big)^k, \qquad \qquad H_2 = \gamma^{-2} \sum_{k \geq 1} \bff_k \Big( \frac{\gamma^2}{1+\gamma^2}\Big)^k.
\end{equation}
An analysis of the sequence $\aff_k$, $\bff_k$ reveals that $|\aff_k| \leq |\aff_4|$, $|\bff_k| \leq |\bff_4|$ whenever $k \geq 4$.

We let $H_1^{(N)}$ and $H_2^{(N)}$ be defined by
$$
H_1^{(N)} =  \gamma^{-2} \sum_{k =1}^N \aff_k \Big( \frac{\gamma^2}{1+\gamma^2}\Big)^k, \qquad \qquad H^{(N)}_2 = \gamma^{-2} \sum_{k =1}^N \bff_k \Big( \frac{\gamma^2}{1+\gamma^2}\Big)^k.
$$
 
\begin{lemma}\label{lem:hs}
    Let $N \geq 1$ be a positive integer\footnote{We are going to choose $N =4$.}, and $\ell_N := \sqrt{a_N^2 + b_N^2}$. We have, for $i = 1, 2$,
\begin{equation}
    H_i^{(N)} - \ell_N \Big(\frac{\gamma^2}{1+\gamma^2} \Big)^N \leq H_i \leq H_i^{(N)} + \ell_N \Big(\frac{\gamma^2}{1+\gamma^2} \Big)^N
\end{equation}
\end{lemma}

\begin{proof}
We have, letting $\beta = \frac{\gamma^2}{1+\gamma^2}$, for $i = 1$ (the case $i = 2$ is the same)
\begin{equation}
    |H_i - H_i^{(N)}| \leq |a_N| \gamma^{-2} \sum_{k \geq N+1} \beta^k =|a_N| \Big( \frac{\gamma^2}{1+\gamma^2}\Big)^N.
\end{equation}
\end{proof}

\subsection{Computing bounds for the Picard iterates symbolically}

In this section, we are going to compute precise bounds for the Picard iterates. in view of~\eqref{eq:masterrhs}, we have to estimate the quantities
\begin{equation}
    -\frac{2}{\pi}\int_0^\infty\Big(H_1 \mathring{w}_1 + H_2 \mathring{w}_2  \Big) d\gamma, \qquad \qquad K_2(t)
\end{equation}
in time.

Let $N_3, N_4 \in \mathbb{N}$, $N_3, N_4 \geq 1$. According to the two previous sections, we bound
\begin{align}
    &H_i^{(N_3)} - \ell_N \Big(\frac{\gamma^2}{1+\gamma^2} \Big)^{N_3} \leq H_i \leq H_i^{(N_3)} + \ell_N \Big(\frac{\gamma^2}{1+\gamma^2} \Big)^{N_3},\\
    &\mathring{w}_1^{(N_4)} + \alpha_{\ell} \leq \mathring{w}_1 \leq \mathring{w}_1^{(N_4)} + \alpha_{u},\\
    &\mathring{w}_2^{(N_4)} + \beta_{\ell} \leq \mathring{w}_2 \leq \mathring{w}_2^{(N_4)} + \beta_{u}.
\end{align}
Here, $\mathring{w}_1^{(N_4)} := \sum_{j = 0}^{N_4} \mathring{c}_j(t) \frac{1}{(1+\gamma^2)^j}$, $\mathring{w}_2^{(N_4)} := \sum_{j = 0}^{N_4} \mathring{d}_j(t) \frac{1}{(1+\gamma^2)^j}$, and in addition $\alpha_\ell, \beta_\ell, \alpha_u, \beta_u$ are defined in Remark~\ref{rmk:alphabetalu}.

We then plug the bounds arising from Lemma~\ref{lem:etas} into the expression for $B_1$ and $B_2$ in Lemma~\ref{lem:alphabeta}. We also pick $N_3 = 4$ and $N_4 = 25$ (in other words, choosing $M = 25$ in lemma~\ref{lem:alphabeta} and $N = 4$ in lemma~\ref{lem:hs}). This gives, considering the term $\int_0^\infty H_2 \mathring{w}_2 d\gamma$, since\footnote{This is checked symbolically.} $\bff_j \leq 0$ for all $j = 0, \ldots, 25$ and $\mathring{w}_2 \geq 0$,
\begin{align}
    &-\frac 2 \pi \int_0^\infty H_2 \mathring{w}_2 d\gamma \\
    &\leq \frac 2 \pi \int_0^\infty \Big(-H^{(N_3)}_2 + \ell_{N_3} \Big(\frac{\gamma^2}{1+\gamma^2} \Big)^{N_3}\Big) \mathring{w}_2 d\gamma\\
    &\leq  -\frac 2 \pi \int_0^\infty H^{(N_3)}_2 (\mathring{w}^{(N_4)}_2 + \beta_u) d\gamma + \frac 2 \pi \int_0^\infty \ell_{N_3} \Big(\frac{\gamma^2}{1+\gamma^2} \Big)^{N_3}(\mathring{w}^{(N_4)}_2 + \beta_u) d\gamma
\end{align}
Then, the terms containing $\ringw_2^{(N_4)}$ are computed symbolically as follows:
\begin{equation}
     -\frac 2 \pi \int_0^\infty H^{(N_3)}_2 \mathring{w}^{(N_4)}_2 d\gamma = -\sum_{k=1}^{N_3} \sum_{j = 0}^{N_4} \bff_k \mathring{d}_j(t) \frac 2 \pi \int_0^\infty \Big(\frac{\gamma^2}{(1+\gamma^2)}\Big)^k \frac{1}{(1+\gamma^2)^j}\frac{ d\gamma}{\gamma^2} = -\sum_{k=1}^{N_3} \sum_{j = 0}^{N_4} \bff_k \mathring{d}_j(t) c^{(j)}_k. 
\end{equation}
Finally, the terms containing $\beta_u$ are estimated as follows. Recall that  
\begin{align}
&\beta_u(t) = 13 \int_0^t \int_0^w \exp(-23(w-s)) (P_{N_4}- Q_{N_4})(s) \frac{1}{(1+e^{8(t-s)}\gamma^2)^{N_4+1}} ds dw \\
&\qquad \qquad + \underbrace{\int_0^t \frac{Q_{N_4}(w)}{(1+ e^{8(t-w)}\gamma^2)^{N_{4}+1}}dw}_{:= (i)},
\end{align}
and that the coefficients $\bff_k$ are negative. Focusing our attention on the terms arising from $(i)$, we have
\begin{align}
    &-\frac 2 \pi \int_0^\infty H^{(N_3)}_2  (i) d\gamma \\
    & =  -\frac 2 \pi \sum_{k=1}^{N_3}\bff_k \int_0^t \int_0^\infty \frac{Q_{N_4}(w)}{(1+ e^{8(t-w)}\gamma^2)^{N_4+1}} \Big(\frac{\gamma^2}{1+\gamma^2}\Big)^k \frac{d\gamma}{\gamma^2} dw\\
    & \leq  -\frac 2 \pi \sum_{k=1}^{N_3}\bff_k \int_0^t  Q_{N_4}(w) \eta_u(8(t-w), N_4+1, k) dw
\end{align}
The other bounds are computed analogously. We collect the results in the following Lemma.

\begin{lemma}\label{lem:picardprecise}
    The Picard iterates $P_n(t)$ of lemma~\ref{lem:picard} satisfy the bounds computed symbolically in the companion file \emph{\texttt{shorttime\_sym.m}}. Moreover, $P_0(0) > 0$ and $P_3(t) > 0$ for all $t \in [0, \log(4)]$.
\end{lemma}

\begin{proof}[Proof of lemma~\ref{lem:picardprecise}]
The bounds have been obtained in the above section. Concerning the lower bound for $P_3(t)$, we first show (symbolically) that the derivative $\Upsilon'(t)$ does not exceed $20000$ in magnitude. With this knowledge, we evaluate (again symbolically) the expression\footnote{After replacing it with an appropriate polynomial lower bound whose coefficients are symbolic fractions.} for $P_3(t)$ on a fine enough grid, which suffices to show that $P_3(t) > 0$ for $t \in [0, \log(4)]$. The details of the symbolic approach are described in section~\ref{sec:shorttimeproof}.
\end{proof}

This in particular shows lemma~\ref{lem:picard}, which concludes the proof.

\subsection*{Data availability statement}

Data sharing not applicable to this article as no datasets were generated or analysed during the current study.

\subsection*{Conflict of interest statement}

The authors have no competing interests to declare that are relevant to the content of this article.

\appendix

\addtocontents{toc}{\protect\setcounter{tocdepth}{0}}

\section{Proofs by symbolic calculation}\label{sec:symbolicproof}
In this section, we outline the parts of several proofs in the paper which require symbolic calculation, with an emphasis on how the companion files are used to obtain the proof.

\subsection{Proof of Lemma~\ref{lem:PQ}}\label{sec:pqproof}

This proof is carried out in the companion file \texttt{PQ\_sym.m}. After some technical preparations, which are described in the file (spanning sections A through E), the main algorithm is contained in Section F. We reduce the computation to showing that a certain polynomial $B(t)$ with rational (symbolic) coefficients is positive for $t \in [0,1]$. We notice that $B(t) = \texttt{pospoly}(t) + \texttt{negpoly}(t)$, where $\texttt{pospoly}(t)$ (resp. $\texttt{negpoly}(t)$) is a polynomial with positive (resp. negative) coefficients. We then implement the following algorithm. We start by showing $\texttt{pospoly}(0) + \texttt{negpoly}(0) > 0$. Then, we increase $t_{\text{inner}}$ by amounts of $1/100$ as long as $\texttt{pospoly}(0) + \texttt{negpoly}(t_{\text{inner}}) > 0$. When $\texttt{pospoly}(0) + \texttt{negpoly}(t_{\text{inner}}) < 0$, we set $t_{\text{outer}} := t_{\text{inner}}$, and we restart the procedure from $t_{\text{outer}}$: we compute $\texttt{pospoly}(t_{\text{outer}}) + \texttt{negpoly}(t_{\text{inner}})$ and keep increasing $t_{\text{inner}}$ while $t_{\text{outer}}$ is kept fixed as long a s $\texttt{pospoly}(t_{\text{outer}}) + \texttt{negpoly}(t_{\text{inner}})>0$. This procedure is then repeated until $t_{\text{inner}}$ surpasses $1$. If that happens, due to the monotonicity of both polynomials, this shows that $B(t)$ is positive between $0$ and $1$. Note finally that, since all computation are exact (they only involve symbolic fractions), this constitutes a rigorous proof.

\subsection{Proof of Lemma~\ref{lem:longtime}} \label{sec:prooflongtime}

The proof is carried out in the companion file \texttt{longtime\_sym.m}. We proceed section by section.

\subsubsection{Section A}
The relevant values are computed which are then fed into the bounds in Section~\ref{sec:linfty}. We provide a legend for the various variable names appearing in Section A. Here $(a_1, \ldots, a_n) \widetilde{\geq} (b_1, \ldots, b_n)$ means $a_k \geq |b_k|$ for all $k = 1, \ldots, n$:
\begin{align}
&( \texttt{F10,
F20,
Gam100,
Gam200,
F10mod,
F20mod,
SG10,
G1bar0,
G2bar0,}\\
& \qquad \texttt{
Gamb10,
Gamb20,
Gam100,
Gam200})\\
& \widetilde{\geq} \  (F_1^{*}(0), F_2^{*}(0), \Gamma^*_1(0), \Gamma^*_2(0), F_1^{*, \text{mod}}(0), F_2^{*, \text{mod}}(0), \int_0^\infty  \bar{F}_1^{*, \text{mod}}(0) \gamma^{-2}d\gamma, \bar{G}_1(0), \bar{G}_2(0), \\
& \qquad \bar{\Gamma}_1^*(0) , \bar{\Gamma}_2^*(0), \Gamma^*_1(0), \Gamma^*_2(0)).
\end{align}

\subsubsection{Section B}
Computation of the bounds in Section~\ref{sec:linfty}. Here (note that on the LHS we have functions of $t$, and the bounds should be interpreted pointwise in $t$),
\begin{align}
&( \texttt{SG1,
G1barpG2bar,
G1bar,
Gamb2,
Gamb1,
Gam20,
Gam10})\\
& \widetilde{\geq} \  (\int_0^\infty G_1 d\gamma, |\bar{G}_1(t)|_\infty + |\bar{G}_2(t)|_\infty, |\bar{G}_1(t)|_\infty , |\bar{\Theta}_2^{\text{mod}}(t)|_\infty, |\bar{\Theta}_1^{\text{mod}}(t)|_\infty, \Theta^{\text{mod}}_2(t, \gamma = 0), \Theta^{\text{mod}}_1(t,\gamma=0)).
\end{align}

\subsubsection{Section C}
Computation of the final bound~\eqref{eq:finalupsilon} (note that $-45.6$ is used as a lower bound for the value of the Laplace transform at $-1$ of $\strokedint \Theta^{\text{mod}}_1$).

\subsubsection{Section D}
Proof that the evolution is positive after time $\log(4)$. After substituting $t \to -\log(t)$, we factor the resulting polynomial $B(t)$:
\begin{equation}
    B(t) = c t B_1(t),
\end{equation}
where $c$ is a positive constant, and $B_1$ is of degree $27$ and has integer coefficients. We write
$
B_1 = \sum_{n=0}^{27} b_n t^n.
$
We proceed to check that
$
b_0 - \max_{t \in [0,1/4]} \sum_{n=1}^{27} |b_n| t^n  \geq b_0 -  \sum_{n=1}^{27} |b_n| 4^{-n} >0
$
which concludes the proof.

\subsection{Proof of Lemma~\ref{lem:picardprecise}}\label{sec:shorttimeproof}

The proof is carried out in the file \texttt{shorttime\_sym.m}. We divide the companion file into several sections, which we describe in detail.

\subsubsection{Section A}
This section serves the purpose of computing the following quantities:
\begin{equation}
    \texttt{af, bf, Ifl, Ifu, Jfl, Jfu}.
\end{equation}
The first two quantities are the exact coefficients in the series expansion of $\bar F^{*}_1$ (resp. $\bar{F}^*_2$), which have already appeared in display~\eqref{eq:afbf}. The other quantities are lower and upper bounds for the following integrals. Let\footnote{Note that indices in Matlab start from $1$.} $j \geq 0$:
\begin{align}
    &\texttt{Ifl(j+1)} \leq \frac 2 \pi \int_0^\infty \frac{\bar F_1^*}{(1+\gamma^2)^j}\frac{d\gamma}{\gamma^2} \leq  \texttt{Ifu(j+1)},\\
    &\texttt{Jfl(j+1)} \leq \frac 2 \pi \int_0^\infty \frac{\bar F_2^*}{(1+\gamma^2)^j}\frac{d\gamma}{\gamma^2} \leq  \texttt{Jfu(j+1)}.
\end{align}

\subsubsection{Section B}
In section B1, we compute the coefficients $\texttt{cring(j)}:= \hat{\mathring{c}}_{j-1}(\xi)$, and $\texttt{dring(j)} := \hat{\mathring{d}}_{j-1}(\xi)$. In section B2, we compute 
$$
\texttt{C(j+1,k)} := c^{(j)}_k = \frac 2 \pi \int_0^\infty \frac{\gamma^{2(k-1)}}{(1+\gamma^2)^{j+k}} d\gamma,
$$
where $j = 0, 1, \ldots$ and $k = 1, 2, \ldots$

In section B3 we compute $P_{N_4}(t)$ and $Q_{N_4}(t)$.

\subsubsection{Section C}
We compute upper and lower bounds for $K_2$:
\begin{equation}
    \texttt{K2l(t)} \leq K_2(t) \leq \texttt{K2u(t)}.
\end{equation}
These bounds are obtained by calling the function $\texttt{alphabeta}$, which uses the bounds for $\alpha$ and $\beta$ in Lemma~\ref{lem:alphabeta}, and the bounds for $\eta$ in Lemma~\ref{lem:etas} in order to estimate 
$$
\int_0^\infty \frac{\beta}{(1+\gamma^2)}d\gamma
$$
from above and from below. The remaining terms are computed symbolically.

\subsubsection{Section D}

We compute upper and lower bounds for the expression
\begin{equation}
    - \int_0^\infty (H_1 \mathring{w}_1 + H_2 \mathring{w}_2) d\gamma
\end{equation}
The bounds are again computed calling the function $\texttt{alphabeta(m,n,P,Q,C)}$, which computes bounds for the integrals\footnote{Note that the dependence on $\texttt{n}$ when calling the function $\texttt{alphabeta}$ refers to the truncation index $N_4$ (truncation of the weight). In our case, we are choosing $\texttt{n} = N_4 = 25$. Moreover, $\texttt{P}$ and $\texttt{Q}$ refer to the functions $P_{(N_4)}$ and $Q_{(N_4)}$, and finally $\texttt{C}$ refers to the matrix $c^{(j)}_k$ for convenience, since such matrix is quite expensive to compute (so, we compute it only once at the beginning).}:
\begin{equation}\label{eq:alphabetadescription}
    \int_{0}^\infty  \Big(\frac{\gamma^2}{1+\gamma^2} \Big)^m \alpha \frac{d\gamma}{
\gamma^2}, \qquad \qquad \int_{0}^\infty \Big(\frac{\gamma^2}{1+\gamma^2} \Big)^m \beta \frac{d\gamma}{
\gamma^2},
\end{equation}

\subsubsection{Section E}
We compute upper and lower bounds for the zeroth iterate of the Volerra equation~\eqref{eq:mainvolte}, that is the function $g(t)$ appearing on the RHS. More precisely, in Section E1 we compute upper and lower bounds for
\begin{equation}
    \hat K_1 \frac{\xi}{\xi+1}\Big(\frac{1}{\xi+6} F_{*,1}(0) + \frac {13}{(\xi+6)(\xi+9)} F_{*,2}(0)\Big).
\end{equation}
The Laplace transform of the lower bound is computed in \texttt{firsll}, and the Laplace transform of the upper bound is computed in \texttt{firstul}.

In Section E2, on the other hand, we compute the lower bound (\texttt{secondPltot}) and the upper bounds (\texttt{secondPutot}) of the inverse Laplace transform of
\begin{equation}
    \frac{\xi}{\xi+1}\Big(1-\frac {130}{(\xi+6)(\xi+9)}\Big) \frac 1 {10} \frac{2}\pi \Big( \frac{130}{\xi(\xi+23)} \mathcal{I} - \int\Big(H_1 \hat{\mathring{w}}_1 + H_2 \hat{\mathring{w}}_2  \Big) d\gamma\Big)
\end{equation}

In Section E3, we consider the lower and upper bounds just obtained, and we replace them with lower and upper bounds whose coefficients are still exact fractions, but whose numerators and denominators are significantly smaller in size. This allows for faster calculations later on.

\subsubsection{Section F}
We compute the Picard iterates of the Volterra equation~\eqref{eq:mainvolte} up until the third iterate.

\subsubsection{Section G}
We show that the Picard iterate $P_1(t)$ is positive for time $t \in [0, \log(3)]$, and that the Picard iterate $P_3(t)$ is positive for time $[\log(3), \log(4)]$ Since in particular $P_1$ lower bounds $P_3$, this shows that $P_3(t)>0$ for all $t \in[0, \log(4)]$. The functions are time-stepped using the following bounds for the derivative (see Remark~\ref{rmk:derivative}):
\begin{align}
    &|\Upsilon'(t)| \leq 20000 \qquad \text{when} \qquad t \in [0, \log(3)],\\
    &|\Upsilon'(t)| \leq 5000 \qquad  \ \text{when} \qquad t \in [\log(3), \log(4)].
\end{align}
Note importantly that the functions\footnote{To be more precise, the lower bounds for these functions, which are computed in the companion file \texttt{shorttime\_sym.m}.} $P_1(t)$ and $P_2(t)$, as computed by the above procedure, are sums of terms of the type $c t^m e^{-nt}$, where $m,n $ are positive integers and $c$ is a constant. In order to be sure that the symbolic calculations are exact, before time-stepping, we replace the lower bounds for the functions $P_1(t)$ and $P_2(t)$ by polynomial lower bounds. This is accomplished by the function \texttt{makepoly(symexpr, N)}, which, given a symbolic expression (\texttt{symexpr}) containing sums of terms of the type $c t^m e^{-nt}$ replaces each of those terms with $c t^m (L_N(t))^n$ if $c>0$, and $c t^m(U_N(t))^n$ if $c < 0$, where $L_N(t)$ (resp.~$U_N(t)$) are polynomial lower (resp.~upper) bounds for $\exp(-t)$ (obtained by taking the Taylor polynomial at $0$ truncated at $N$). In addition, after carrying out the above procedure, the function \texttt{makepoly} replaces the above polynomial with yet another polynomial lower bound whose (symbolic, rational) coefficients have smaller denominators than the original expression (to speed up calculations).

\subsubsection{Section H} We compute the functions \texttt{eta} in Section H1 and \texttt{alphabeta} in section H2. Given $\texttt{m, n}$ (and \texttt{C}, which is passed to the function for convenience), the function $\texttt{eta(m,n,C)}$ returns two functions $\texttt{low}$ and $\texttt{upp}$ (of the symbolic variable $\texttt{tt}$), which are respectively lower and upper bounds for the expression
\begin{equation}
    \int_0^\infty \Big(\frac{\gamma^2}{1+\gamma^2}\Big)^{\texttt{n}} \frac{1}{(1+\exp(8 \texttt{tt}))^{\texttt{m}}} \frac{d\gamma}{\gamma^2}
\end{equation}
computed according to Lemma~\ref{lem:etas}. The function \texttt{alphabeta} has already been described in display~\eqref{eq:alphabetadescription}.

\section{Symbolic bounds}\label{sec:appb}
In this section, we first compute $L^\infty$ bounds for the profile and some associated quantities, and then we compute precise values for integrals involving the profile.

\subsection{Computation of $L^\infty$ bounds}\label{sec:linftyapp}
We use 
$$
\bar \Gamma_i = S^i_N + E^i_N,$$
with $i \in \{1,2\}$, where $S_N := \sum_{j=1}^N a_j \Big(\frac{\gamma^2}{1+\gamma^2} \Big)^j$, to get

\begin{align}
&|\bar \Gamma_1|_\infty \leq |S_N^1|_\infty + |E_N^1|_\infty \leq |S_N^1|_\infty + 10 M^{1.1} N^{-0.1} \sqrt{|a_M|^2 + |b_M|^2}.\\
&|\bar \Gamma_2|_\infty \leq |S_N^2|_\infty + |E_N^2|_\infty \leq |S_N^2|_\infty  + 10 M^{1.1} N^{-0.1} \sqrt{|a_M|^2 + |b_M|^2}.
\end{align}
We pick $M= 5000$ and $N = 5000$. A symbolic calculation (contained in the companion file \texttt{profile.m}, section A) yields
\begin{equation}
    |E_N^1|, |E_N^2|_\infty \leq 1/2.
\end{equation}
On the other hand, again by a symbolic calculation,
\begin{equation}
    |S^1_N|_{\infty} \leq 49.5, \qquad \qquad |S^2_N|_{\infty} \leq 33, 
\end{equation}
which finally yields
\begin{equation}
    |\bar\Gamma_1|_\infty \leq 50, \qquad \qquad |\bar \Gamma_2|_\infty \leq 33.5.
\end{equation}
Since $\bar\Gamma_1 \leq 0$ and $\bar \Gamma_2 \leq 0$, together with $\Gamma_1(0) \in (44,45)$ and $\Gamma_2(0) \in (20,21)$, we also have
\begin{equation}
    |\Gamma_1|_\infty \leq 50, \qquad \qquad | \Gamma_2|_\infty \leq 33.5.
\end{equation}
In addition,
\begin{equation}
    |\bar G_1^{(\Gamma_1)}|_\infty \leq |a(1)| + |a(2)| \leq 32, \qquad \qquad |\bar G_2^{(\Gamma_1)}|_\infty   \leq |b(1)| + |b(2)| \leq 27.
\end{equation}

We turn to estimating the $L^\infty$ norms of $\bar F^{1, \text{mod}}_*$ and $\bar F^{2, \text{mod}}_*$. Recall that the constant $c = 237/46$. We also recall that, by the profile equation,
\begin{align}
    | \bar F^{1, \text{mod}}_* |_\infty \leq |-(c+1/2) S_N^1 + 13/2 S_N^2| + (c+1/2)|E_N^1|_\infty + 13 |E_N^2|_\infty \leq 157 + 10 = 167.
\end{align}
Also,
\begin{align}
    | \bar F^{2, \text{mod}}_* |_\infty \leq \Big| \frac{5}{1+\gamma^2}\Big(S_N^1 + \gamma^2 A\Big)\Big|_\infty + 5|E_N^1|_\infty  + (c+1/2)|E_N^2|_\infty \leq 110 + 3 = 113.
\end{align}
Let $G^{(F^{\text{mod}}_*)} := \frac{\bar{F}^{\text{mod}}_*}{\gamma^2}$.  Since the coefficients $\text{afm}_k$ and $\text{bfm}_k$ of the functions $ F^{1, \text{mod}}_*$ (resp  $F^{2, \text{mod}}_*$) satisfy $|\text{afm}_k| \leq |\text{afm}_2|$ for $k \geq 2$, and $|\text{bfm}_k| \leq |\text{bfm}_2|$ for $k \geq 2$, we have\footnote{This is easily seen by the fact that if $f(\gamma) = \sum_{k \geq 1} r_k \Big(\frac{\gamma^2}{1+\gamma^2} \Big)^k$, then $\overline{\bar{f}/\gamma^2} = - \frac{\gamma^2}{1+\gamma^2} r_1 + \frac{1}{1+\gamma^2}\sum_{k \geq 1}r_k \Big(\frac{\gamma^2}{1+\gamma^2} \Big)^k$.}
\begin{equation}
    |\bar G_1^{(F^{\text{mod}}_*)}|_\infty \leq |\text{afm}_1| + |\text{afm}_2|\leq 68, \qquad \qquad |\bar G_2^{(F^{\text{mod}}_*)}|_\infty \leq |\text{bfm}_1| + |\text{bfm}_2| \leq 147. 
\end{equation}
Finally, we need bounds on $L_\theta F^{\text{mod}}_* = L_\theta F_*$. We have $F_* = (- \frac 12 \Gamma_1^* + \frac{13}2 \Gamma_2^*, \frac{5}{1+\gamma^2}(\Gamma_1^* - 26)-\frac 12 \Gamma_2^*)$, so that
\begin{equation}
    L_\theta F^* = \Big(- \frac{39}{2} \Gamma_2, \quad  \frac{1}{1+\gamma^2}  \big( - 40 (\Gamma_1-26)\gamma^2  + 15 \Gamma_1 - 9 \cdot 5 \cdot 26 + 13\cdot 5 \strokedint \Gamma_2 \big)\Big)
\end{equation}
Letting $\zeta:= L_\theta F^*$, we then have\footnote{Note that these estimates can be dramatically improved, but they are only needed to check positivity at the gridpoints, so we won't do it here.}
\begin{align}
    &|\zeta_1|_\infty \leq \frac{39} 2 |\Gamma_2|_\infty \leq 654,\\
    &|\zeta_2|_\infty \leq 55|\Gamma_1|_\infty + 26 \cdot 40 + 45 \cdot 26 + 13 \cdot 5 \cdot |\Gamma_2|_\infty \leq 7200,\\
    &|\bar \zeta_1|_\infty \leq \frac{39} 2 |\bar \Gamma_2|_\infty \leq 654,\\
    &|\bar \zeta_2|_\infty \leq 70|\bar \Gamma_1|_\infty + 26 \cdot 40 + 45 \cdot 26 + 13 \cdot 5 \cdot |\Gamma_2|_\infty \leq 8000,\\
    &|\bar G^{(\zeta)}_1|_\infty \leq \frac{39} 2 |\bar G^{(\Gamma)}_2|_\infty \leq 527,\\
    &|\bar G^{(\zeta)}_2|_\infty \leq  40|\bar \Gamma_1|_\infty + |A| + 15 |\Gamma_1|_\infty + 15 |\bar \Gamma_1|_\infty + 15 |G_1^{(\Gamma)}|_\infty + 45 \cdot 26 + 13 \cdot 5 \cdot |\Gamma_2|_\infty \leq 8100,
\end{align}
We record the bounds obtained so far in the following lemma.
\begin{lemma}
    The functions $F^{*, \text{mod}}_1$,  $F^{*, \text{mod}}_2$, and $L_\theta F^{*, \text{mod}}$ satisfy the $L^\infty$ bounds in the file \emph{\texttt{profile.m}}, section A.
\end{lemma}

\subsection{Computation of integrals of the profile with error bounds}\label{sec:IJ}
Let us now obtain an expression for the quantities
$$
I_j := \frac 2 \pi \int \frac{\bar \Gamma_1}{\gamma^2} \frac 1 {(1+\gamma^2)^j}d\gamma, \qquad \qquad J_j := \frac 2 \pi \int \frac{\bar \Gamma_2}{\gamma^2} \frac 1 {(1+\gamma^2)^j}d\gamma,
$$
with $j \geq 0$, $j \in \N$.

We have
\begin{equation}
    \frac{\bar \Gamma_1}{\gamma^2} = \sum_{k = 1}^\infty a_k \frac{\gamma^{2k-2}}{(1+\gamma^2)^k}.
\end{equation}
This implies that, fixing $N_1 > 0$,
\begin{equation}
    I_j = \sum_{k = 1}^{N_1} a_k \frac 2 \pi \int_0^\infty \frac{\gamma^{2k-2}}{(1+\gamma^2)^{k+j}} + E_j,
\end{equation}
with $E_j := \sum_{N_1+1}^\infty a_k \frac 2 \pi \int_0^\infty \frac{\gamma^{2k-2}}{(1+\gamma^2)^{k+j}}$.

Letting
\begin{equation}
    c^{(j)}_k := \frac 2 \pi \int_0^\infty \Big(\frac{\gamma^2}{1+\gamma^2} \Big)^k\frac{1}{(1+\gamma^2)^j} \frac{1}{\gamma^2} d\gamma,
\end{equation}
we calculate\footnote{Note that this expression is a multiple of the well-known \emph{beta function}.}
\begin{equation}
   c_k^{(j)} = \frac 1 \pi \frac{\Gamma(1/2 + j) \Gamma( k- 1/2)}{ \Gamma(j + k)}
\end{equation}
where $\Gamma$ denotes the Gamma function.

We compute
$
c_1^{(0)} = 1.
$
For $j = 0$, $k > 1$, we have 
$
c_k^{(0)} = \frac{k-\frac 32}{k-1} c_{k-1}^{(0)}.
$
For $j > 0$, $k \geq 0$
$
c_k^{(j)} = \frac{j - \frac 12}{j+k-1}c_{k}^{(j-1)}.
$
Let $N > 0$, we split the contribution of these integrals into 
$
I^{\leq N}_j , I^{>N}_j , J^{\leq N}_j, J^{> N}_j 
$
truncating the series to index $N$. Using the expression for $c_k^{(j)}$, together with the bounds for $a_j$ and $b_n$ obtained previously in Lemma~\ref{lem:spgap}, we deduce
\begin{equation}
    \frac 2 \pi | I^{>N}_j |, \frac 2\pi |J^{>N}_j| \leq  (|a_N|^2 + |b_N|^2)^{\frac 12} c_N^{(j-1)} .
\end{equation}

A symbolic calculation, contained in \texttt{profile.m} (section B2) implies the following lemma, concerning $I_j$ and $J_j$.

\begin{lemma}
The quantities $I_j$, $J_j$ satisfy the bounds in the file \emph{\texttt{profile.m}}, section B2, for $j \in \{0, 1, \ldots, 34\}$:
\begin{align}
    &\emph{\texttt{Il(j+1)}} \leq I_j \leq \emph{\texttt{Iu(j+1)}},\\
    &\emph{\texttt{Jl(j+1)}} \leq J_j \leq \emph{\texttt{Ju(j+1)}}.
\end{align}
\end{lemma}

We consider the quantities
\begin{align}
&I^{F_*}_j :=\frac 2 \pi \int \frac{\bar F^{*}_1}{\gamma^2} \frac 1 {(1+\gamma^2)^j}d\gamma, \qquad \qquad J^{F_*}_j := \frac 2 \pi \int \frac{\bar F^{*}_2}{\gamma^2} \frac 1 {(1+\gamma^2)^j}d\gamma,\\
&I^{F_*, \text{mod}}_j :=\frac 2 \pi \int \frac{\bar F^{*, \text{mod}}_1}{\gamma^2} \frac 1 {(1+\gamma^2)^j}d\gamma, \qquad \qquad J^{F_*, \text{mod}}_j := \frac 2 \pi \int \frac{\bar F^{*, \text{mod}}_2}{\gamma^2} \frac 1 {(1+\gamma^2)^j}d\gamma.
\end{align}
From the profile equation, we also have
\begin{align}
    &I^{F_*, \text{mod}}_j = \frac{13}2 J_j - \Big(\frac 12 + c \Big) I_j,\\
    &J^{F_*, \text{mod}}_j = 5\Big(I_{j+1} + A c^{(j)}_1\Big) -  \Big(\frac 12 + c \Big) J_j,
\end{align}
A symbolic calculation then leads to the following lemma.
\begin{lemma}\label{lem:IJmod}
    The quantities $I^{F_*, \text{mod}}_j$ and $J^{F_*, \text{mod}}_j$ satisfy the bounds in the companion file \emph{\texttt{profile.m}}, section B2:
\begin{align}    
    &\emph{\texttt{Ifl(j+1)}} \leq I^{F_*}_j \leq \emph{\texttt{Ifu(j+1)}},\\
    &\emph{\texttt{Jfl(j+1)}} \leq J^{F_*}_j \leq \emph{\texttt{Jfu(j+1)}},\\
    &\emph{\texttt{Ifmodl(j+1)}} \leq I^{F_*, \text{mod}}_j \leq \emph{\texttt{Ifmodu(j+1)}},\\
    &\emph{\texttt{Jfmodl(j+1)}} \leq J^{F_*, \text{mod}}_j  \leq \emph{\texttt{Jfmodu(j+1)}}.
\end{align}
    
\end{lemma}

\section{Guide to the companion files}\label{sec:companion}
In this Section, we list the companion files and describe their content, along with their inter-dependencies. We divide the list into three main files and four additional auxiliary files.

\vspace{10pt}

\noindent {\bf Main files:}
\begin{itemize}
    \item \texttt{profile.m}: we compute the bounds relevant to the large time argument (the $L^\infty$ bounds on $F^*$ and the profile used in section~\ref{sec:linfty}), and the bounds on the integrals containing $F^*$ and the profile (which are described in section~\ref{sec:appb}). These bounds are instrumental to virtually all calculations contained in \texttt{longtime\_sym.m} and \texttt{shorttime\_sym.m}.
    \item \texttt{longtime\_sym.m}: we perform all calculations pertaining to Section~\ref{sec:largetime}. The two main tasks accomplished are the $L^\infty$ bound on the evolution of $\Theta^{\text{mod}}(t)$ (described in section~\ref{sec:linfty}), and the bound on the Laplace transform of  $\Theta^{\text{mod}}$ at $\xi = -1$ (described in section~\ref{sec:laplacem1}). This file depends on the bounds computed in \texttt{profile.m}
    \item \texttt{shorttime\_sym.m}: we carry out the calculations relevant to Section~\ref{sec:shorttime}. In particular, this file depends on the the bounds for the values of integrals of $F^*$ computed in \texttt{profile.m}. In practice, when running \texttt{shorttime\_sym.m}, the necessary values are computed by calling \texttt{build\_shorttime\_data.m}. In addition, \texttt{shorttime\_sym.m} depends on the bounds on the derivative $\Upsilon'(t)$ computed in \texttt{derivative\_sym.m}, on the fact that the expression $P(t) - Q(t)$ is positive (shown in \texttt{PQ\_sym.m}), and finally on the procedure carried out in \texttt{makepoly.m}, which serves the purpose of converting the bounds computed by Picard iteration into polynomial with rational (symbolic) coefficients (which are computed exactly via symbolic calculation).
\end{itemize}
{\bf Auxiliary files:}
\begin{itemize}
    \item \texttt{derivative\_sym.m}: in this file, we show the bounds on $\Upsilon'(t)$ required in the time-stepping part of the argument in \texttt{shorttime\_sym.m}.
    \item \texttt{PQ\_sym.m}: we show that the expression $P(t) - Q(t)$ is positive for all times. This is used in the estimates in section~\ref{sec:etas} (and the relevant symbolic computations are carried out in \texttt{shorttime\_sym.m}).
    \item \texttt{makepoly.m}: given a symbolic expression which is composed by terms of the type $\frac{p}{q} t^m e^{- nt}$, where $p/q$ is a symbolic fraction, and $m, n \in \N$, this function computes a polynomial lower bound for the expression by replacing $e^{-t}$ with a polynomial bound.
    \item \texttt{build\_shorttime\_data.m}: computes the values (integrals of $F^*$, associated quantities, and coefficients of the expansion of the profile) which are required to run \texttt{shorttime\_sym.m}.
\end{itemize}

\section{Some miscellaneous facts}\label{sec:appd}
In this section, we collect useful facts and lemmas.

\begin{lemma}\label{lem:prodottino}
    Suppose that a sequence of non-negative real numbers $\{x_k\}_{k \in \N}$ satisfies, for $k \geq M_1$
    \begin{equation}
        x_{k+1} \leq \Big(1- \frac{\alpha}{k}\Big) x_{k}
    \end{equation}
    Where $\alpha > 0$. Then, if $M_2 \geq M_1$, and if $k \geq M_2$,
    \begin{equation}
        x_k \leq x_{M_2} \big(\frac{k}{M_2} \big)^{-\alpha}
    \end{equation}
\end{lemma}
\begin{proof}
    We have
    \begin{equation}
        \log(x_k) = \log(x_{M_2}) + \sum_{j = M_2}^{k-1} \log\Big(1- \frac{\alpha}{k}\Big) \leq \log(x_{M_2}) -\alpha \sum_{j = M_2}^{k-1}  \frac{1}{k} \leq \log(x_{M_2}) -\alpha \int_{M_2}^k\frac{1}{x}dx
    \end{equation}
    and we conclude.
\end{proof}

\section{Notation}

We recall here the notation for the main objects in this paper as a guide the interested reader.

\begin{itemize}
    \item $\mathfrak{L}f := f + z\p_z f + \scrl f$,
    \item $\scrl$: Defined in~\cite{EP2023}, Section 4.1,
    \item $\mathscr{L}_0 q := z \partial_{z} q+\frac{1}{1+z} \mathcal{L} q+\frac 12 F_{*} \int_{z}^{\infty} \strokedint \frac{q}{(1+y)^{2}} d y$,
    \item $\mathcal{L}(\Gamma_1, \Gamma_2) =L_\theta \Gamma - \frac{\Gamma^*}{2} \strokedint \Gamma_1$,
    \item $L_\theta \Gamma = \Big(2 D_\theta \Gamma_1 + 6 \Gamma_1 - \Gamma_2, \quad 2 D_\theta \Gamma_2 + 9 \Gamma_2 - 130 \cos^2 \theta \Big(\Gamma_1 - \strokedint \Gamma_1 \Big) \Big)$,
    \item $\Gamma^*$: Solution of $ L_\theta \Gamma^* = 0$ with $\strokedint \Gamma_1 = 2$,
    \item $\strokedint f = \frac{2}{\pi}\int_0^{\frac \pi 2} f(\theta) d\theta$,
    \item $D_\theta f = \sin(2\theta)\p_\theta f$,
    \item $F^*_1 = D_\theta \Gamma^*_1 + \frac 5 2 \Gamma^*_1, \qquad \qquad F^*_2 = D_\theta \Gamma^*_2 + 4 \Gamma^*_2$,
    \item $\hat{K}_1(\xi) := \frac{(\xi+13)(\xi+10)}{\xi+23}\hat{\kappa}(\xi)$,
    \item $\hat{\kappa}(\xi) := \mathscr{F}(\kappa(r))$ where $\kappa(r) = \frac 2 \pi \int_0^\infty \Big(\frac{w_2(r)}{1+\gamma^2} \Big) d \gamma$,
    \item $K_2(t) = \frac{2}\pi \int_0^\infty \ringw_2 \frac{d\gamma}{(1+\gamma^2)}$,
        \item $\Upsilon(t) := \strokedint \exp(- t \mathcal{L}) F_* d\theta$.
\end{itemize}

\bibliographystyle{plain}
\bibliography{bibliography.bib}

\end{document}